\documentclass[12pt]{amsart}

\usepackage{amssymb}
\usepackage{amsthm}
\usepackage{amsmath}
\usepackage{graphicx}
\usepackage{subfigure}
\usepackage{amssymb}
\usepackage[capitalize]{cleveref}

\theoremstyle{plain}
\newtheorem*{thm*}{Theorem}
\newtheorem{thm}{Theorem}

\newtheorem{lem}[thm]{Lemma}
\newtheorem{cor}[thm]{Corollary}
\newtheorem{prob}{Problem}
\newtheorem{conj}{Conjecture}

\theoremstyle{definition}
\newtheorem{defn}[thm]{Definition}
\newtheorem*{notation}{Notation}
\newtheorem{claim}{Claim}
\newtheorem{claiminproof}{Claim}[thm]

\newcommand{\real}{\mathbb{R}}
\newcommand{\imag}{\mathbb{I}}
\newcommand{\complex}{\mathbb{C}}
\newcommand{\disk}{\mathbb{D}}
\newcommand{\id}{\mathrm{id}}
\newcommand{\diam}{\mathrm{diam}}
\newcommand{\proj}{\mathrm{proj}}

\newcommand{\sm}{\setminus}
\newcommand{\0}{\emptyset}
\newcommand{\e}{\varepsilon}
\newcommand{\len}{\mathsf{len}}
\newcommand{\midpt}{\mathsf{m}}
\newcommand{\lam}{\mathcal{L}}
\newcommand{\g}{\mathfrak{g}}
\newcommand{\h}[1]{\widehat{#1}}
\newcommand{\til}[1]{\widetilde{#1}}
\newcommand{\bol}[1]{\mathbf{#1}}
\newcommand{\hull}{\mathrm{Hull}}
\newcommand{\Equi}{\mathrm{Equi}}
\newcommand{\arcends}{\mathrm{Ends}}
\newcommand{\wE}{\widetilde{\mathrm{exp}}}

\begin{document}

\title{Extension of isotopies in the plane}
\author[Hoehn, Oversteegen, and Tymchatyn]{L. C. Hoehn \and L. G. Oversteegen \and E. D. Tymchatyn}

\address[L.\ C.\ Hoehn]{Nipissing University, Department of Computer Science \& Mathematics, 100 College Drive, Box 5002, North Bay, Ontario, Canada, P1B 8L7}
\email{loganh@nipissingu.ca}

\address[L.\ G.\ Oversteegen]{University of Alabama at Birmingham, Department of Mathematics, Birmingham, AL 35294, USA}
\email{overstee@uab.edu}

\address[E.\ D.\ Tymchatyn]{University of Saskatchewan, Department of Mathematics and Statistics, 106 Wiggins road, Saskatoon, Canada, S7N 5E6}
\email{tymchat@math.usask.ca}

\thanks{The first named author was partially supported by NSERC grant RGPIN 435518}
\thanks{The third named author was partially supported by NSERC grant OGP-0005616}


\begin{abstract}
Let $A$ be any plane set. It is known that a holomorphic motion $h: A \times \disk \to \complex$ always extends to a holomorphic motion of the entire plane.  It was recently shown that any isotopy $h: X \times [0,1] \to \complex$, starting at the identity, of a plane continuum $X$ also extends to an isotopy of the entire plane.  Easy examples show that this result does not generalize to all plane compacta.  In this paper we will provide a characterization of isotopies of uniformly perfect plane compacta $X$ which extend to an isotopy of the entire plane.  Using this characterization, we prove that such an extension is always possible provided the diameters of all components of $X$ are uniformly bounded away from zero.
\end{abstract}

\maketitle

\section{Introduction} 
Denote the complex plane by $\complex$ and the open unit disk by $\disk$.  An \emph{isotopy} of a set $X \subset \complex$ is a homotopy $h: X \times [0,1] \to \complex$ such that for each $t \in [0,1]$, the function $h^t: X \to \complex$ defined by $h^t(x) = h(x,t)$ is an embedding (i.e.\ a homeomorphism of $X$ onto the range of $h^t$).

Suppose that $h: X \times [0,1] \to \complex$ is an isotopy of a compactum $X \subset \complex$ such that $h^0 = \id_X$.  We consider the old problem of when the isotopy $h$ can be extended to an isotopy of the entire plane \footnote{We are indebted to Professor R.\ D.\ Edwards who communicated a related problem to us.}.

A more restrictive variant of the notion of an isotopy is a holomorphic motion (see e.g.\ \cite{astamart01}).  The remarkable $\lambda$-Lemma \cite{slod91} states that any holomorphic motion of any plane set can be extended to a holomorphic motion of the entire plane.  See \cite{astamart01} for a different proof and some history of that problem.

Although the $\lambda$-Lemma holds for arbitrary plane  sets, some additional restrictions are needed for the existence of an extension of an isotopy to the entire plane $\complex$.  First, it is reasonable to restrict to isotopies of plane compacta.  This by itself is not enough since Fabel has shown that there exists an isotopy of a convergent sequence which cannot be extended over the plane (see \cite[p.~991]{fabe05}).  On the other hand, it was shown recently in \cite{ot10} that any isotopy of an arbitrary plane continuum $X$ can be extended over the plane.  In this case each complementary domain $U$ of $X$ is simply connected and, hence, there exists a conformal isomorphism $\varphi_U: \disk \to U$.  The proof made use of two key analytic results for these conformal isomorphisms: the Carath\'{e}odory kernel convergence theorem, and the Gehring-Hayman inequality for the diameters of hyperbolic geodesics in $U$.

Let us now consider the case when $X$ is a plane compactum.  Since we may assume that $X$ contains at least three points, the boundary of every complementary component $U$ of $X$ contains at least three points, so $U$ is hyperbolic, i.e.\ there exists an analytic covering map $\varphi_U: \disk \to U$ (see \cite{Ahlfors1973}).

There is an analogue of the Carath\'{e}odory kernel convergence theorem which holds for families of analytic covering maps (see \cref{sec:analytic covering maps}).  For an analogue of the Gehring-Hayman inequality, an additional geometric condition will be required:

\begin{defn}
A compact subset $X \subset \complex$ is \emph{uniformly perfect with constant $k$} provided there exists $0 < k < 1$ so that for all $r < \diam(X)$ and all $x \in X$,
\[ \{z \in \complex : kr \leq |z - x| \leq r\} \cap X \neq \0 .\]
\end{defn}

Clearly every uniformly perfect set is perfect and the standard ``middle-third'' Cantor set is uniformly perfect.  It is known that the Gehring-Hayman estimate on the diameter of hyperbolic geodesics still holds for very analytic covering map $\varphi_U: \disk \to U$ to a domain $U$ whose boundary is uniformly perfect (see \cref{sec:analytic covering maps} for details).

The main result in this paper is a characterization of isotopies $h: X \times [0,1] \to \complex$ of uniformly perfect plane compacta $X$ which can be extended over the entire plane (see \cref{main1}).  We use our characterization to prove that any isotopy of a plane compactum such that the diameter of every component is uniformly bounded away from zero can be extended over the plane (see \cref{main2}).  Along the way, we will provide simpler proofs of some of the technical results in \cite{ot10}.

\subsection{Notation}
\label{sec:notation}

By a \emph{map} we mean a continuous function.  For $z \in \complex$, the magnitude of $z$ is denoted $|z|$, so that the Euclidean distance between two points $z,w \in \complex$ is $|z - w|$.  Given $z_0 \in \complex$ and $r > 0$, denote
\[ B(z_0,r) = \{z \in \complex: |z_0 - z| < r\} .\]

By a \emph{domain} we mean a connected, open, non-empty set $U \subset \complex$.  If $X \subset \complex$ is closed, then a \emph{complementary domain} of $X$ is a component of $\complex \sm X$.  A \emph{crosscut} of a domain $U$ is an \emph{open arc} $Q$ (i.e.\ $Q \approx (0,1) \subset \real$) contained in $U$ such that $\overline{Q}$ is a \emph{closed arc} (i.e.\ $\overline{Q} \approx [0,1]$) whose endpoints are in $\partial U$.  Note that the endpoints of $\overline{Q}$ are required to be distinct.  In general, if $A$ is an open arc whose closure $\overline{A}$ is a closed arc, we may refer to the endpoints of $\overline{A}$ as the ``endpoints of $A$''.

A \emph{path} is a map $\gamma: [0,1] \to \complex$.  Given a domain $U$, we say $\gamma$ is a \emph{path in $U$} if $\gamma((0,1)) \subset U$.  Note that \emph{we allow the possibility that $\gamma(0) \in \partial U$ and/or $\gamma(1) \in \partial U$} -- we still call such a path a path in $U$.

We will make frequent use of covering maps in this paper.  Given a covering map $\varphi: V \to U$, where $V$ and $U$ are domains, a \emph{lift} of a point $x \in U$ is a point $\h{x} \in V$ such that $\varphi(\h{x}) = x$.  Similarly, if $\gamma$ is a path with $\gamma([0,1]) \subset U$ then a lift of $\gamma$ is a path $\h{\gamma}$ in $V$ such that $\varphi \circ \h{\gamma} = \gamma$.

The \emph{Hausdorff metric} $d_H$ measures the distance between two compact sets $A_1,A_2 \subset \complex$ as follows:
\[ d_H(A_1,A_2) = \max \{ \max_{z_1 \in A_1} \min_{z_2 \in A_2} |z_1 - z_2|, \; \max_{z_2 \in A_2} \min_{z_1 \in A_1} |z_1 - z_2| \} .\]
Equivalently, $d_H(A_1,A_2)$ is the smallest number $\e \geq 0$ such that $A_1$ is contained in the closed $\e$-neighborhood of $A_2$ and $A_2$ is contained in the closed $\e$-neighborhood of $A_1$.

Given an isotopy $h: X \times [0,1] \to \complex$, we denote $h^t = h|_{X \times \{t\}}$ and, for $x \in X$, we denote $x^t = h^t(x)$.

\section{Preliminaries}
\label{sec:preliminaries}

In this section we collect several tools which we use in this paper.  Many of these are standard analytical results; others are less well-known.

\subsection{Bounded analytic covering maps}
\label{sec:analytic covering maps}

It is a standard classical result (see e.g.\ \cite{Ahlfors1973}) that for any domain $U \subset \complex$ whose complement contains at least two points, and for any $z_0 \in U$, there is a complex analytic covering map $\varphi: \disk \to U$ such that $\varphi(0) = z_0$.  Moreover, this covering map $\varphi$ is uniquely determined by the argument of $\varphi'(0)$.

Many of the results below hold only for analytic covering maps $\varphi: \disk \to U$ to bounded domains $U$.  For the remainder of this subsection, let $U \subset \complex$ be a bounded domain, and let $\varphi_U=\varphi: \disk \to U$ be an analytic covering map.

\begin{thm}[Fatou \cite{Fatou06}]
\label{Fatou}
The radial limits $\lim_{r \to 1^-} \varphi(re^{i\theta})$ exist for all points $e^{i\theta}$ in $\partial \disk$ except possibly for a set of linear measure zero.
\end{thm}

From now on, we will always assume that any bounded analytic covering map $\varphi: \disk \to U$ has been extended to be defined over all points $e^{i\theta} \in \partial \disk$ where the radial limit exists by $\varphi(e^{i\theta}) = \lim_{r \to 1^-} \varphi(re^{i\theta})$.  Note that the function $\varphi$ is not necessarily continuous at these points.

For this extended map $\varphi$, we extend the notion of lifts.  If $\gamma$ is a path in $U$ (recall this allows for the possibility that $\gamma(0)$ and/or $\gamma(1)$ belongs to $\partial U$), then a \emph{lift} of $\gamma$ is a path $\h{\gamma}$ in $\disk$ such that $\varphi \circ \h{\gamma} = \gamma$.  This means that if $\gamma(0) \in \partial U$, then $\h{\gamma}(0) \in \partial \disk$ and $\varphi$ is defined at the point $\h{\gamma}(0)$ (and $\varphi(\h{\gamma}(0)) = \gamma(0)$); and likewise for $\gamma(1)$ and $\h{\gamma}(1)$.

\begin{thm}[Riesz \cite{Riesz16, Riesz23}]
\label{Riesz}
For each $x \in \partial U$, the set of points $e^{i\theta}$ for which $\lim_{r\to 1^-} \varphi(re^{i\theta}) = x$ has linear measure zero in $\partial \disk$.
\end{thm}

The next result about lifts of paths is very similar to classical results for covering maps.  Since our extended map $\varphi_U$ is not a covering map at points in $\partial \disk$, we include a proof for completeness.

\begin{thm}
\label{lift}
Suppose $\gamma$ is a path in $U$ such that $\gamma((0,1]) \subset U$.  Let $\h{z} \in \disk$ be such that $\varphi(\h{z}) = \gamma(1)$.  Then there exists a unique lift $\h{\gamma}$ of $\gamma$ with $\h{\gamma}(1) = \h{z}$.

In particular, if $\gamma(0) \in \partial U$, then $\h{\gamma}(0) \in \partial \disk$, $\varphi$ is defined at $\h{\gamma}(0)$ (i.e.\ the radial limit of $\varphi$ exists there), and $\varphi(\h{\gamma}(0)) = \gamma(0)$.
\end{thm}

\begin{proof}
We may assume that $\gamma(0) \in \partial U$.  Since $\varphi$ is a covering map, $\gamma|_{(0,1]}$ lifts to a path with initial point $\h{z}$ which compactifies on a continuum $K \subset \partial \disk$.  If $K$ is non-degenerate, then there exists by \cref{Fatou} a set $E$ of positive measure in the interior of $K$ so that for each $e^{i\theta} \in E$, the radial limit $\lim_{r \to 1^-} \gamma(r e^{i\theta})$ exists.  Since the graph of $\h{\gamma}$ compactifies on $K$ we can choose a sequence $s_i \to 1$ so that $\h{\gamma}(s_i)= r_i e^{i\theta}$ with $r_i \to 1$.  It follows that the radial limit $\lim_{r \to 1^-} \varphi^t(r e^{i\theta}) = \gamma(1)$ for each $e^{i\theta} \in E$, a contradiction with \cref{Riesz}.  Thus $K$ is a point $e^{i\theta}$.

If $\gamma(0)$ is a limit point of $\partial U$, then we can choose arbitrarily small $\rho > 0$ so that the circle $S(\gamma(0),\rho) = \partial B(\gamma(0),\rho)$ intersects $\partial U$, and $\varphi(0), \gamma(1) \notin B(\gamma(0),2\rho)$.  Let $C$ be the component of $S(\gamma(0),\rho) \sm \partial U$ so that the closure of the component of $\gamma([0,1]) \cap B(\gamma(0),\rho)$, which contains $\gamma(0)$, intersects $C$.  By the above, $C$ lifts to a crosscut $\h{C}$ of $\disk$ such that $e^{i\theta}$ is contained in the component $H$ of $\disk \sm \h{C}$ which does not contain $0$.  Since a terminal segment of the radial segment $\{r e^{i\theta}: 0 \leq r < 1\}$ is contained in $H$, and $\rho$ is arbitrarily small, it follows that $\varphi(\h{\gamma}(0)) = \lim_{r\to 1^-} \varphi(r e^{i\theta}) = \gamma(0)$ as required.

In the case that $\gamma(0)$ is an isolated point of $\partial U$ (this case will not be needed in this paper as all domains we consider will have perfect boundaries), a similar argument can be made by lifting a small circle in $U$ centered at $\gamma(0)$.  We leave this case to the reader.
\end{proof}

The next result is a variant of \cref{lift}, in which the base point of the path to be lifted is in the boundary of $U$.

In the case that the boundary of $U$ is uniformly perfect, we prove below in \cref{Hlift} a stronger result about lifting a homotopy under covering maps to a domain whose boundary is changing under an isotopy.  The present result can be obtained as a Corollary to \cref{Hlift} by using the identity isotopy.  We omit a proof for the non-uniformly perfect case, since we won't need it for this paper.

\begin{thm}
\label{lift from bd}
Suppose $\gamma$ is a path in $U$ such that $\gamma((0,1]) \subset U$ and $\gamma(0) \in \partial U$.  Let $\h{x} \in \partial \disk$ be such that $\varphi_U(\h{x}) = \gamma(0)$ and $\gamma$ is homotopic to the radial path $\varphi_U|_{\{r \h{x} \;:\; 0 \leq r \leq 1\}}$ under a homotopy in $U$ that fixes the point $\gamma(0)$.  Then there exists a lift $\h{\gamma}$ of $\gamma$ with $\h{\gamma}(0) = \h{x}$.  Moreover, if $\partial U$ is perfect, this lift $\h{\gamma}$ is unique.
\end{thm}

The \emph{hyperbolic metric} on the unit disk $\disk$ is given by the form $\frac{2|dz|}{1 - |z|^2}$, meaning that the length of a smooth path $\gamma: [0,1] \to \disk$ is $\int_0^1 \frac{2 |\gamma'(t)|}{1 - |\gamma(t)|^2} \,dt$.  The important property of the hyperbolic metric for us is that (hyperbolic) geodesics in $\disk$ are pieces of round circles or straight lines which cross the boundary $\partial \disk$ orthogonally.  Via the covering map $\varphi: \disk \to U$, we obtain the \emph{hyperbolic metric on $U$}, in which the length of a smooth path in $U$ is equal to the length of any lift of that path under $\varphi$ -- this length is independent of the choice of lift.  It is a standard result that the hyperbolic metric on $U$ is independent of the choice of covering map $\varphi: \disk \to U$.

\begin{thm}[Gehring-Hayman \cite{PR98,pomm02}]
\label{gehrhay}
Suppose $\partial U$ is uniformly perfect with constant $k$.  There exists a constant $K$ such that if $\h{g}$ is a hyperbolic geodesic in $\disk$ and $\h{\Gamma}$ is a curve with the same endpoints as $\h{g}$, then
\[ \diam(\varphi(\h{g})) \leq K \cdot \diam(\varphi(\h{\Gamma})) .\]

The constant $K$ depends only on $k$, not on the domain $U$ itself or on the choice of analytic covering map $\varphi$.
\end{thm}

We end this subsection with a discussion of analytic covering maps of varying domains in the plane.  We will make use of the notion of Carath\'{e}odory kernel convergence, which was introduced by Carath\'{e}odory for univalent analytic maps in \cite{cara12}, then extended by Hejhal to the case of analytic covering maps.

Let $U_1,U_2,\ldots$ and $U_\infty$ be domains and let $z_1,z_2,\ldots$ and $z_\infty$ be points with $z_n \in U_n$ for all $n = 1,2,\ldots$ and $z_\infty \in U_\infty$.  We say that \emph{$\langle U_n, z_n \rangle \to \langle U_\infty, z_\infty \rangle$ in the sense of Carath\'{e}odory kernel convergence} provided that (i) $z_n \to z_\infty$; (ii) for any compact set $K \subset U_\infty$, $K \subset U_n$ for all but finitely many $n$; and (iii) for any domain $U$ containing $z_\infty$, if $U \subseteq U_n$ for infinitely many $n$, then $U \subseteq U_\infty$.

\begin{thm}[\cite{hej74}; see also \cite{Comerford2013}]
\label{caratheodory}
Let $U_1,U_2,\ldots$ and $U_\infty$ be domains and let $z_1,z_2,\ldots$ and $z_\infty$ be points with $z_n \in U_n$ for all $n = 1,2,\ldots$ and $z_\infty \in U_\infty$.  Let $\varphi_\infty: \disk \to U_\infty$ be the analytic covering map such that $\varphi(0) = z_\infty$ and $\varphi'(0) > 0$.  Likewise, for each $n = 1,2,\ldots$, let $\varphi_n: \disk \to U_n$ be the analytic covering map such that $\varphi_n(0) = z_n$ and $\varphi_n'(0) > 0$.  Then $\langle U_n, z_n \rangle \to \langle U_\infty, z_\infty \rangle$ in the sense of Carath\'{e}odory kernel convergence if and only if $\varphi_n \to \varphi_\infty$ uniformly on compact subsets of $\disk$.
\end{thm}

\subsection{Partitioning plane domains}
\label{sec:partitioning domains}

Let $U$ be a bounded domain in $\complex$.  We next describe a way of partitioning $U$ into simple sets which are either circular arcs or regions whose boundaries are unions of circular arcs.

Let $\mathcal{B}$ be the collection of all open disks $B(c,r) \subset U$ such that $|\partial B(c,r) \cap \partial U| \geq 2$.  Let $\mathcal{C}$ be the collection of all centers of such disks, and for $c \in \mathcal{C}$ let $r(c)$ be the  radius of the corresponding disk in $\mathcal B$.  The set $\mathcal{C}$, called the \emph{skeleton of $U$}, was studied by several authors (see for example \cite{fre97}).  Note that for each $c \in \mathcal{C}$, $B(c,r(c))$ is conformally equivalent with the unit disk $\disk$ and, hence, can be endowed with the hyperbolic metric $\rho_c$.  Let $\hull(c)$ be the convex hull of the set $\partial B(c,r(c)) \cap \partial U$ in $B(c,r(c))$ \emph{using the hyperbolic metric $\rho_c$ on the disk $B(c,r(c))$}.  The following theorem by Kulkarni and Pinkall generalizes an earlier result by Bell \cite{bell76} (see \cite{bfmot13} for a more complete description):

\begin{thm}[\cite{kulkpink94}]
\label{KP}
For each $z \in U$ there exists a unique $c \in \mathcal{C}$ such that $z \in \hull(c)$.
\end{thm}

\begin{figure}
\begin{center}
\includegraphics{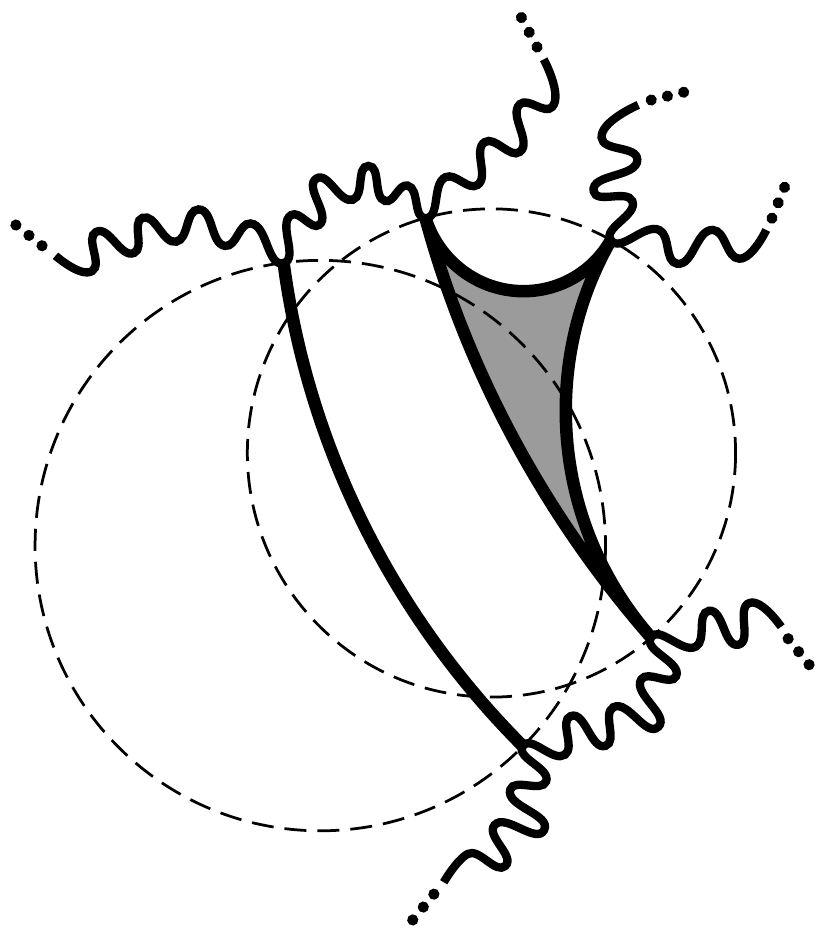}
\end{center}

\caption{Depiction of two examples of the sets $\hull(c)$ from the Kulkarni-Pinkall decomposition of a domain $U$ in $\complex$.  In the picture, $U$ is a component of the complement of the wavy lines.}
\label{fig:KP}
\end{figure}

Let $\mathcal{J}$ be the collection of all crosscuts of $U$ which are contained in the boundaries of the sets $\hull(c)$ for $c \in \mathcal{C}$.  The elements of $\mathcal{J}$ are circular open arcs (called \emph{chords}) whose endpoints are in $\partial U$.  Two such chords do not cross each other inside $U$ (i.e., if $\ell \neq \ell'$ are chords in $\mathcal{J}$, then $\ell \cap \ell' = \0$) and the limit of any convergent sequence of chords in $\mathcal{J}$ is either a chord in $\mathcal{J}$ or a point in $\partial U$.  In particular, the subcollection of chords of diameter greater or equal to $\e$ is compact for each $\e > 0$.  As such, the family $\mathcal{J}$ is close to being a \emph{lamination} of $U$ (see \cref{d:lamination} in \cref{sec:characterization} below).  However, it is possible that uncountably many distinct chords in $\mathcal{J}$ have the same pair of endpoints $x,y \in \partial U$.

\subsection{Equidistant sets}
\label{sec:equi set}

Let $A_1$ and $A_2$ be disjoint closed sets in $\complex$.  The \emph{equidistant set} between $A_1$ and $A_2$ is the set
\[ \Equi(A_1,A_2) = \left\{ z \in \complex: \min_{w \in A_1} |z-w| = \min_{w \in A_2} |z-w| \right\} .\]

The equidistant set is a convenient way to define a set running ``in between'' $A_1$ and $A_2$.  Moreover, it has a very simple local structure in the case that the sets $A_1$ and $A_2$ are not ``entangled'' in the sense of the following definition:

\begin{defn}
We say that $A_1$ and $A_2$ are \emph{non-interlaced} if whenever $B(c,r)$ is an open disk contained in the complement of $A_1 \cup A_2$, there are disjoint arcs $C_1,C_2 \subset \partial B(c,r)$ such that $A_1 \cap \partial B(c,r) \subset C_1$ and $A_2 \cap \partial B(c,r) \subset C_2$.  We allow for the possibility that $C_1 = \0$ in the case that $A_2 \cap \partial B(c,r) = \partial B(c,r)$, and vice versa.
\end{defn}

By a \emph{$1$-manifold} in the plane, we mean a \emph{closed} set $M \subset \complex$ such that each component of $M$ is homeomorphic either to $\real$ or to $\partial \disk$, and these components are all open in $M$.

\begin{thm}[\cite{brou05,aartbrouwover}]
\label{thm:manifold}
Let $A_1$ and $A_2$ be disjoint closed sets in $\complex$.  If $A_1$ and $A_2$ are non-interlaced, then $\Equi(A_1,A_2)$ is a $1$-manifold in the plane.
\end{thm}

\subsection{Midpoints of paths}
\label{sec:midpoints}

We identify the space of all paths in the plane $\complex$ with the function space $\mathcal{C}([0,1],\complex)$ with the \emph{uniform metric}; that is, the distance between two paths $\gamma_1,\gamma_2 \in \mathcal{C}([0,1],\complex)$ is equal to $\sup \{|\gamma_1(t) - \gamma_2(t)|: t \in [0,1]\}$.

The standard Euclidean length of a path is not a well-behaved function from $\mathcal{C}([0,1],\complex)$ to $[0,\infty)$.  First, it is not defined (i.e., not finite) for all paths in $\mathcal{C}([0,1],\complex)$, but only for rectifiable paths.  Second, paths can be arbitrarily close in the uniform metric and yet have very different Euclidean path lengths.

However, there do exist alternative ``path length'' functions $\len: \mathcal{C}([0,1],\complex) \to [0,\infty)$ such that $\len$ is defined for \emph{all} paths in $\mathcal{C}([0,1],\complex)$, and $\len$ is continuous with respect to the uniform metric on $\mathcal{C}([0,1],\complex)$ and the standard topology on $[0,\infty) \subset \real$, see \cite{Morse1936,Silverman1969,hot13}.  Such an alternative path length function can be used to define a choice of ``midpoint'' of a path which varies continuously with the path.  Specifically, the midpoint of $\gamma$ is defined to be the point $\midpt(\gamma) = \gamma(t_0)$, where $t_0 \in (0,1)$ is chosen such that $\len(\gamma|_{[0,t_0]}) = \len(\gamma|_{[t_0,1]})$.

In this paper, we will not need to know any particulars about the definitions of such path length functions, but only this result about existence of such midpoints, which we state below.

\begin{thm}[see e.g.\ \cite{hot13}]
\label{thm:midpt}
There is a continuous function
\[ \midpt: \mathcal{C}([0,1],\complex) \to \complex \]
such that $\midpt(\gamma) \in \gamma((0,1))$ for all $\gamma \in \mathcal{C}([0,1],\complex)$.

Moreover, if $\gamma_1$ and $\gamma_2$ are both parameterizations of a closed arc $A$ (i.e.\ if $\gamma_1([0,1]) = \gamma_2([0,1]) = A$ and $\gamma_1$ and $\gamma_2$ are homeomorphisms between $[0,1]$ and $A$), then $\midpt(\gamma_1) = \midpt(\gamma_2)$.
\end{thm}

In light of the second part of \cref{thm:midpt}, given an (open or closed) arc $A$, we define the midpoint of $A$ to be $\midpt(A) = \midpt(\gamma)$ where $\gamma$ is any path which parameterizes $A$ ($\overline{A}$ if $A$ is an open arc).

\section{Main Theorem}
\label{sec:characterization}

In this section, we state and prove the main theorem of this paper, which is a characterization of isotopies of uniformly perfect plane compacta which can be extended over the entire plane.  Note that the example of Fabel mentioned in the Introduction can easily be modified to obtain an isotopy $h: X \times [0,1] \to \complex$ so that for each $t$, $X^t = h^t(X)$ is a uniformly perfect Cantor set with the same constant $k$.  Thus, additional assumptions are required to ensure the extension of such an isotopy over the plane.

\begin{thm}
\label{main1}
Suppose that $h: X \times [0,1] \to \complex$ is an isotopy of a compactum $X \subset \complex$ starting at the identity, such that $X^t$ is uniformly perfect with the same constant $k$ for each $t \in [0,1]$.  Then the following are equivalent:
\begin{enumerate}
\item $h$ extends to an isotopy of the entire plane $\complex$;
\item For each $\e > 0$ there exists $\delta > 0$ such that for any crosscut $Q$ of a complementary domain $U$ of $X$ with $\diam(C) < \delta$, there exists a homotopy $h_Q: (X \cup Q) \times [0,1] \to \complex$ starting at the identity which extends $h$ and is such that $h_Q^t(Q) \cap X^t = \0$ and $\diam(h_Q^t(Q)) < \e$ for all $t \in [0,1]$.
\end{enumerate}
\end{thm}

It is trivial to see that condition (i) implies condition (ii) from \cref{main1}.

To obtain the converse, we will in fact prove a stronger characterization in \cref{main1technical} below.  To state this Theorem, we introduce the following simple condition:

\begin{defn}
Let $X \subset \complex$ be a compact set and let $h: X \times [0,1] \to \complex$ be an isotopy of $X$ starting at the identity.  We say that $X$ is \emph{encircled} if $X$ has a component which is a large circle $\Sigma$ such that $h^t |_\Sigma$ is the identity for all $t \in [0,1]$, and $X^t \sm \Sigma$ is contained in a compact subset of the bounded complementary domain of $\Sigma$ for all $t \in [0,1]$.
\end{defn}

Note that if (ii) from \cref{main1} holds, then we may additionally assume without loss of generality (i.e.\ without falsifying condition (ii) from \cref{main1}) that $X$ is encircled.

\subsection{Tracking bounded complementary domains}
\label{sec:tracking domains}

For the remainder of this section, we assume that $h: X \times [0,1] \to \complex$ is an isotopy of a compact set $X \subset \complex$ starting at the identity, such that $X^t$ is uniformly perfect for all $t \in [0,1]$ with the same constant $k$, and that $X$ is encircled.

Clearly such an isotopy can be extended over the unbounded complementary domain of $X$ as the identity for all $t \in [0,1]$.  Hence we only need to consider bounded complementary domains for the remainder of this section.

Let $U$ be a bounded complementary domain of $X$.  Choose a point $z_U \in U$.  Clearly the isotopy $h$ can be extended to an isotopy $h_U: (X \cup \{z_U\}) \times [0,1] \to \complex$ starting at the identity.  Define $U^t$ to be the complementary domain of $X^t$ which contains the point $h_U^t(z_U) = z_U^t$.  Let $\varphi_U^t: \disk \to U^t$ be the analytic covering map such that $\varphi_U^t(0) = z_U^t$ and $(\varphi_U^t)'(0) > 0$.  It is straightforward to see that if $t_n \to t_\infty$, then the pointed domains $\langle U^{t_n}, z_U^{t_n} \rangle$ converge to $\langle U^{t_\infty}, z_U^{t_\infty} \rangle$ in the sense of Carath\'{e}odory kernel convergence.  Hence, by \cref{caratheodory}, the covering maps $\varphi_U^{t_n}$ converge to $\varphi_U^{t_\infty}$ uniformly on compact subsets of $\disk$.  We will always assume that the complementary domains $U^t$ of $X^t$ and analytic covering maps $\varphi_U^t: \disk \to U^t$ are defined in this way.  It is clear that this definition of $U^t$ does not depend on the choices of $z_U$ and $h_U$.

The following Theorem is a stronger characterization of isotopies of uniformly perfect plane compacta that can be extended over the plane than the one given in \cref{main1}, in the sense that condition (ii) of \cref{main1technical} is weaker than condition (ii) of \cref{main1}.  We will in fact use this stronger characterization in \cref{sec:large components}.

\begin{thm}
\label{main1technical}
Suppose that $h: X \times [0,1] \to \complex$ is an isotopy of a compactum $X \subset \complex$ starting at the identity, such that $X^t$ is uniformly perfect with the same constant $k$ for each $t \in [0,1]$, and that $X$ is encircled.  Then the following are equivalent:
\begin{enumerate}
\item $h$ extends to an isotopy of the entire plane $\complex$;
\item For each bounded complementary domain $U$ of $X$ and each $\e > 0$ there exists $\delta > 0$ with the following property:

For any crosscut $Q$ in $U$ with endpoints $a,b \in \partial U$ and with $\diam(Q) < \delta$, there exists a family $\{\gamma_t: t \in [0,1]\}$ such that (1) $\gamma_t$ is a path in $U^t$ joining $a^t$ and $b^t$ for each $t \in [0,1]$, (2) $\gamma_0$ is homotopic to $Q$ in $U$ with endpoints fixed, (3) $\diam(\gamma_t([0,1])) < \e$ for all $t \in [0,1]$, and (4) there are lifts $\h{\gamma}_t$ of the paths $\gamma_t$ under $\varphi_U^t$ such that the sets $\h{\gamma}_t([0,1])$ vary continuously in $t$ with respect to the Hausdorff metric.
\end{enumerate}
\end{thm}

We have deliberately chosen to use subscripts in the notation for $\gamma_t$ (instead of superscripts like $\gamma^t$) to emphasize the point that the paths $\gamma_t$ are \emph{not} required to change continuously in the sense of an isotopy or homotopy -- we only require the weaker condition that the images of the lifts $\h{\gamma}_t$ vary continuously with respect to the Hausdorff metric.  Even though condition (ii) of \cref{main1technical} is more cumbersome to state, we demonstrate in \cref{sec:large components} that it is easier to apply.

The proofs of \cref{main1} and \cref{main1technical} will be completed in \cref{sec:proof main1} below.

\subsection{Lifts in moving domains}
\label{sec:lifts moving}

As in \cref{sec:tracking domains}, we continue to assume that $h: X \times [0,1] \to \complex$ is an isotopy of a compact set $X \subset \complex$ starting at the identity, such that $X^t$ is uniformly perfect for all $t \in [0,1]$ with the same constant $k$, and that $X$ is encircled.

We begin by proving two statements about lifts under the covering maps $\varphi_U^t$, in the spirit of the results from \cref{sec:analytic covering maps} above.

\begin{lem}
\label{smallup}
Let $U$ be a bounded complementary domain of $X$.  For every $\e > 0$ there exists $\delta > 0$ such that for any $t \in [0,1]$ if $\gamma$ is a path in $U^t$ with $\diam(\gamma([0,1])) < \delta$ and $\h{\gamma}$ is any lift of $\gamma$ under $\varphi_U^t$, then $\diam(\h{\gamma}([0,1])) < \e$.
\end{lem}

\begin{proof}
Suppose the lemma fails.  Then there exists $\e > 0$, a sequence $\gamma_i$ of paths in $U^{t_i}$ and lifts $\h{\gamma}_i$ such that $\lim \diam(\gamma_i([0,1])) = 0$ and $\diam(\h{\gamma}_i([0,1])) \geq \e$ for all $i$.  Choose two points $\h{a}_i, \h{b}_i$ in $\h{\gamma}_i([0,1])$ such that $|\h{a}_i - \h{b}_i| > \frac{\e}{2}$, and let $\h{\g}_i$ be the hyperbolic geodesic with endpoints $\h{a}_i$ and $\h{b}_i$.  Put $\varphi_U^{t_i}(\h{\g}_i) = \g_i$.  By \cref{gehrhay}, $\diam(\g_i) \to 0$.  Since the geodesics $\h{\g}_i$ are pieces of round circles or straight lines which cross $\partial \disk$ perpendicularly and have diameter bigger than $\frac{\e}{2}$, there exist $\eta > 0$ and points $\h{x}_i \in \h{\g}_i$ such that $|\h{x}_i| \leq 1 - \eta$ for all $i$.  By choosing a subsequence we may assume that $t_i \to t_\infty$, $\h{x}_i \to \h{x}_\infty \in \disk$, and $\lim \g_i = z_\infty$ is a point in $\overline{U^{t_\infty}}$.  Let $K_i$ be the component of $\h{\g}_i\cap B(\h{x}_\infty,\frac{\eta}{2})$ containing the point $\h{x}_i$.  We may assume $K_i \to K_\infty$, where $K_\infty$ is a non-degenerate continuum in $\disk$.  Since $\varphi_U^{t_i} \to \varphi_U^{t_\infty}$ uniformly on compact sets in $\disk$, $\varphi_U^{t_\infty}(K_\infty) = z_\infty$, which is a contradiction since $\varphi_U^{t_\infty}$ is a covering map.
\end{proof}

Given a homotopy $\Gamma: [0,1] \times [0,1] \to \complex$ we denote for each $t \in [0,1]$ $\Gamma^t = \Gamma|_{[0,1]\times\{t\}}: [0,1] \to \complex$.

\begin{lem}
\label{Hlift}
Let $U$ be a bounded complementary domain of $X$.  Suppose that $\Gamma: [0,1] \times [0,1] \to \complex$ is a homotopy with $\Gamma^t(0) = h^t(\Gamma^0(0)) \in \partial U^t$ and $\Gamma^t(s) \in U^t$ for all $s \in (0,1]$ and all $t \in [0,1]$.  Let $\h{z} \in \disk$ be such that $\varphi_U^0(\h{z}) = \Gamma^0(1)$.  Then there exists a homotopy $\h{\Gamma}: [0,1] \times [0,1] \to \overline{\disk}$ lifting $\Gamma$, i.e.\ $\varphi_U^t \circ \h{\Gamma}^t = \Gamma^t$ for all $t \in [0,1]$, and such that $\h{\Gamma}^0(1) = \h{z}$.
\end{lem}

\begin{proof}
Define $\Psi: \disk \times [0,1] \to \bigcup_{t \in [0,1]} (U^t \times \{t\})$ by $\Psi(z,t) = (\varphi_U^t(z),t)$ for $t \in [0,1]$ and $z \in \disk$.

\begin{claiminproof}
\label{claim:Psi covering}
$\Psi$ is a covering map.
\end{claiminproof}

\begin{proof}[Proof of \cref{claim:Psi covering}]
\renewcommand{\qedsymbol}{\textsquare (\cref{claim:Psi covering})}
Let $(y_0,t_0) \in U^{t_0} \times \{t_0\}$.  Choose a small simply connected neighborhood $V$ of $y_0$ and $\delta > 0$ such that $\overline{V} \cap X^t = \0$ and $V$ is evenly covered by $\varphi_U^t$ for all $t$ with $|t - t_0| \leq \delta$.  Hence, $V \times (t_0-\delta, t_0+\delta)$ is a simply connected neighborhood of $(y_0,t_0)$ in $\bigcup_{t \in [0,1]} (U^t \times \{t\})$.

Next let $(x_0,t_0) \in \Psi^{-1}((y_0,t_0))$.  Since the covering maps $\varphi_U^t$ are uniformly convergent on compact sets, it is not difficult to see that there exists a map $g: (t_0-\delta, t_0+\delta) \to \disk \times [0,1]$ such that $g(t_0) = (x_0,t_0)$ and $\Psi \circ g(t) = (y_0,t)$ for all $t$ with $|t - t_0| < \delta$.

For each $t$ with $|t - t_0| < \delta$, let $x \in U^t$ be such that $g(t) = (x,t)$, and let $W^t$ be the component of $(\varphi_U^t)^{-1}(V)$ which contains the point $x$.  Let $W = \bigcup_{t \in (t_0-\delta, t_0+\delta)} (W^t \times \{t\})$.  Then it is not difficult to see that $\Psi |_W: W \to V \times (t_0-\delta, t_0+\delta)$ is a homeomorphism.  Thus $\Psi$ is a covering map.
\end{proof}

Define $\alpha: [0,1] \times [0,1] \to \bigcup_{t \in [0,1]} (U^t \times \{t\})$ by $\alpha(s,t) = (\Gamma^t(s),t)$.  Define the lift $\h{\alpha}$ of $\alpha$ under $\Psi$ as follows: first lift $\alpha |_{\{1\} \times [0,1]}$, using the covering map $\Psi$, to define $\h{\alpha} |_{\{1\} \times [0,1]}$ such that $\h{\alpha}(1,0) = (\h{z},0)$.  Next, for each $t \in [0,1]$, use \cref{lift} to lift $\alpha |_{[0,1] \times \{t\}}$ to define $\h{\alpha} |_{[0,1] \times \{t\}}$, so that this lift coincides with the first lift of $\alpha |_{\{1\} \times [0,1]}$ at $(1,t)$.  Finally, define $\h{\Gamma} = \pi_1 \circ \h{\alpha}$, where $\pi_1$ denotes the first coordinate projection.

Observe that for all $s \in (0,1]$, the function $\h{\alpha} |_{[s,1] \times [0,1]}$ is the unique lift of $\alpha |_{[s,1] \times [0,1]}$ under the covering map $\Psi$ with $\h{\alpha}(1,0) = \h{z}$, hence is continuous by standard covering map theory.  It follows that $\h{\alpha}$, and hence $\h{\Gamma}$, is continuous on $(0,1] \times [0,1]$.  It remains to prove that $\h{\Gamma}$ is continuous at all points of the form $(0,t_0)$.

Fix $t_0 \in [0,1]$ and $\e > 0$.  Choose $\delta > 0$ small enough (using \cref{smallup}) so that for any $t \in [0,1]$ and any open arc $D$ in $U^t$ of diameter less than $\delta$, each lift $\h{D}$ of $D$ under $\varphi_U^t$ has diameter less than $\frac{\e}{3}$.

Choose $\eta_1,\eta_2 > 0$ small enough so that:
\begin{enumerate}
\item $|\h{\Gamma}^{t_0}(0) - \h{\Gamma}^{t_0}(\eta_1)| < \frac{\e}{3}$ (this is possible since the lifted path $\h{\Gamma}^{t_0}$ is continuous);
\item $|\h{\Gamma}^{t}(\eta_1) - \h{\Gamma}^{t_0}(\eta_1)| < \frac{\e}{3}$ for each $t \in [t_0-\eta_2, t_0+\eta_2]$ (this is possible since we already know that $\h{\Gamma}$ is continuous on $(0,1] \times [0,1]$); and
\item $\Gamma([0,\eta_1] \times [t_0-\eta_2, t_0+\eta_2]) \subset B(\Gamma^{t_0}(0),\frac{\delta}{2})$ (this is possible since $\Gamma$ is continuous).
\end{enumerate}

Now for any $t \in [t_0-\eta_2, t_0+\eta_2]$, the image $\Gamma^t([0,\eta_1])$ has diameter less than $\delta$, hence $\h{\Gamma}^t([0,\eta_1])$ has diameter less than $\frac{\e}{3}$.  It follows that $\h{\Gamma}^t([0,\eta_1]) \subset B(\h{\Gamma}^{t_0}(0), \e)$.  So $[0,\eta_1) \times (t_0-\eta_2, t_0+\eta_2)$ is a neighborhood of $(0,t_0)$ which is mapped by $\h{\Gamma}$ into $B(\h{\Gamma}^{t_0}(0), \e)$.  Thus $\h{\Gamma}$ is continuous at $(0,t_0)$.
\end{proof}

Observe that in light of \cref{Hlift}, condition (ii) of \cref{main1} is stronger than condition (ii) of \cref{main1technical}.  Therefore to complete the proofs of both \cref{main1} and \cref{main1technical}, we must prove that if condition (ii) of \cref{main1technical} holds then the isotopy $h$ extends to the entire plane $\complex$.  Hence we will assume for the remainder of this section that condition (ii) of \cref{main1technical} holds.

\begin{notation}[$\h{a}^t$]
Let $\h{a} \in \partial \disk$ be any point at which $\varphi_U$ is defined (i.e.\ at which the radial limit of $\varphi_U$ exists).  Using any sufficiently small crosscut $Q$ in $U$ which has one endpoint equal to $a = \varphi_U(\h{a})$ and which is the image of a crosscut of $\disk$ having one endpoint equal to $\h{a}$, we obtain from condition (ii) of \cref{main1technical} a family of paths $\{\gamma_t: t \in [0,1]\}$ and lifts $\h{\gamma}_t$ with the properties listed there, and such that $\gamma_t(0) = a^t$ for each $t \in [0,1]$, and $\h{\gamma}_0(0) = \h{a}$.  Because the sets $\h{\gamma}_t([0,1])$ vary continuously in $t$ with respect to the Hausdorff metric, the endpoint $\h{\gamma}_t(0)$ moves continuously in $t$.  Now we define $\h{a}^t = \h{\gamma}_t(0)$ for each $t \in [0,1]$.  Then $\h{a}^0 = \h{a}$ and $\varphi_U(\h{a}^t) = a^t$ for all $t \in [0,1]$.  It is straightforward to see that this definition of $\h{a}^t$ is independent of the choice of crosscut $Q$ and of the paths $\gamma_t$ and lifts $\h{\gamma}_t$ afforded by condition (ii) of \cref{main1technical}.  Thus, in the presence of condition (ii) of \cref{main1technical}, we can extend the superscript $t$ notation to points in $\partial \disk$ at which $\varphi_U$ is defined.  We will assume this is done for all such points $\h{a} \in \partial \disk$ for the remainder of this section.
\end{notation}

\subsection{Hyperbolic laminations}
\label{sec:laminations}

The following condition on a set of hyperbolic geodesics $\lam$ is inspired by a similar notion introduced by Thurston (cf.\ \cite{thur85}).

\begin{defn}
\label{d:lamination}
A \emph{hyperbolic lamination} $\lam$ in a bounded domain $U \subset \complex$ is a closed set of pairwise disjoint hyperbolic geodesic crosscuts in $U$ such that two distinct crosscuts in $\lam$ are disjoint and have \emph{at most one} common endpoint in the boundary of $U$ and the family of crosscuts in $\lam$ of diameter greater or equal $\e$ is compact for any $\e > 0$.

We denote by $\bigcup \lam$ the union of all the crosscuts in $\lam$.  A \emph{gap} of $\lam$ is the closure of a component of $U \sm \bigcup \lam$.
\end{defn}

The compactness condition in \cref{d:lamination} is equivalent to the following statement: if $\langle \g_n \rangle_{n=1}^\infty$ is a sequence of elements of $\lam$, then either $\diam(\g_n) \to 0$, or there is a convergent subsequence whose limit is also an element of $\lam$.

Fix a bounded complementary domain $U$ of $X$.  Recall the Kulkarni-Pinkall construction described in \cref{sec:partitioning domains}: we consider the collection $\mathcal{B}$ of all open disks $B(c,r) \subset U$ such that $|\partial B(c,r) \cap \partial U| \geq 2$.  For each such disk $B(c,r)$, $\hull(c)$ denotes the convex hull of the set $\partial B(c,r(c)) \cap \partial U$ in $B(c,r(c))$ \emph{using the hyperbolic metric $\rho_c$ on the disk $B(c,r(c))$}.  Let $\mathcal{J}$ be the collection of all crosscuts of $U$ which are contained in the boundaries of the sets $\hull(c)$ for $B(c,r) \in \mathcal{B}$.

Let
\[ \h{\mathcal{J}} = \{\h{Q}: \h{Q} \textrm{ is a component of } \varphi_U^{-1}(Q) \textrm{ for some } Q \in \mathcal{J}\} .\]
For any $Q \in \mathcal{J}$, it is straightforward to see that each component $\h{Q}$ of $\varphi_U^{-1}(Q)$ is an open arc whose closure is mapped homeomorphically onto $\overline{Q}$ by $\varphi_U$.

Given an (open) arc $A$, we denote the set of endpoints of $A$ by $\arcends(A)$; that is, $\arcends(A) = \{a,b\}$ means that $a$ and $b$ are the endpoints of (the closure of) $A$.  Let $\mathcal{J}_\arcends = \{\arcends(Q): Q \in \mathcal{J}\}$, and let $\h{\mathcal{J}}_\arcends = \{\arcends(\h{Q}): \h{Q} \in \h{\mathcal{J}}\}$.  These are sets of (unordered) pairs.

For each $t \in [0,1]$, let
\begin{align*}
\h{\lam}^t = \{\h{\g}^t: & \h{\g}^t \textrm{ is the hyperbolic geodesic in } \disk \\
& \textrm{ joining } \h{a}^t,\h{b}^t, \textrm{ where } \{\h{a},\h{b}\} \in \h{\mathcal{J}}_\arcends\}
\end{align*}
and let
\[ \lam^t = \{\varphi_{U}^t(\h{\g}^t): \h{\g}^t \in \h{\lam}^t\} .\]

Observe that $\lam^0$ is the collection of all hyperbolic geodesic crosscuts of $U^0 = U$ which are homotopic (with endpoints fixed) to some crosscut in $\mathcal{J}$.  For $t > 0$, the collection $\lam^t$ is obtained from $\lam^0$ by following the motion of the endpoints of the arcs in $\lam^0$ under the isotopy and joining the resulting points in $\partial U^t$ by the hyperbolic geodesic crosscut $\g^t = \varphi_{U}^t(\h{\g}^t)$ in $U^t$ using the hyperbolic metric induced by $\varphi_U^t$.  We do \emph{not} consider a Kulkarni-Pinkall style partition of the domain $U^t$ for $t > 0$.

We shall prove that $\lam^t$ is a hyperbolic lamination in $U^t$ for each $t \in [0,1]$.  We start with the following lemma.

\begin{lem}
\label{gtarc}
For any $t \in [0,1]$ and any $\h{\g}^t \in \h{\lam}^t$, the map $\varphi_U^t$ is one-to-one on $\h{\g}^t$ and, hence, the corresponding element $\g^t = \varphi_U^t(\h{\g}^t) \in \lam^t$ is a crosscut in $U^t$.  Moreover, if $\g_1^t,\g_2^t$ are two distinct elements of $\lam^t$, then $\g_1^t \cap \g_2^t = \0$ (though their closures may have at most one common endpoint in $\partial U^t$).
\end{lem}

\begin{proof}
Let $\g^0$ be an arbitrary hyperbolic crosscut of $\lam^0$ with endpoints, $a$ and $b$.  By the discussion at the end of \cref{sec:lifts moving}, we can lift $\g^0$ to geodesics $\h{\g}^t$ with continuously varying endpoints.  Let $\h{a}^t$ ($\h{b}^t$) be the endpoints of $\h{\g}^t$ corresponding to $a^t$ ($b^t$, respectively).  Since $\g^0$ is an arc, all components $\h{\g}^0$ of $\varphi_U^{-1}(\g^0)$ are pairwise disjoint geodesic crosscuts of $\disk$.  Since the endpoints of all these crosscuts move continuously in $t$ and the points $a^t$ and $b^t$ are distinct, the geodesics $\h{\g}^t$ are also pairwise disjoint open arcs for all $t$.  Hence, $\varphi_U^t$ is one-to-one on each of these crosscuts and their common image is a geodesic arc $\g^t$.  By a similar argument, the lifts $\h{\g}_1^t$ and $\h{\g}_2^t$ of two distinct geodesics $\g_1^t$ and $\g_2^t$ in $\lam^t$ are pairwise disjoint in $\disk$ and, hence, $\g_1^t \cap \g_2^t = \0$.  It follows easily from the construction that two distinct geodesics in $\lam^0$ share at most one common endpoint and, hence, the same is true for $\lam^t$.
\end{proof}

To prove $\lam^t$ is a hyperbolic lamination in $U^t$ for each $t \in [0,1]$, it remains to show that the collection of arcs in $\lam^t$ of diameter at least $\e$ is compact for every $\e > 0$.  This will follow from the next Lemma, which states that even for varying $t$, the limit of a convergent sequence of elements of the corresponding $\lam^t$ collections must belong to the limit $\lam^t$ collection as well.

\begin{lem}
\label{gtconverge}
Let $\{a_1,b_1\}, \{a_2,b_2\}, \ldots$ be a sequence of pairs in $\mathcal{J}_\arcends$ such that $a_n \to a_\infty$ and $b_n \to b_\infty$, where $a_\infty$ and $b_\infty$ are distinct points in $\partial U$.  Then $\{a_\infty,b_\infty\} \in \mathcal{J}_\arcends$.

Furthermore, let $t_1,t_2,\ldots \in [0,1]$ be a sequence such that $t_n \to t_\infty \in [0,1]$.  For each $n \in \{1,2,\ldots\} \cup \{\infty\}$ and each $t \in [0,1]$, let $\g_n^t \in \lam^t$ be the geodesic with endpoints $a_n^t$ and $b_n^t$.  Then $\g_n^{t_n} \to \g_\infty^{t_\infty}$ in the sense that there exist homeomorphisms $\theta_n: \g_\infty^{t_\infty} \to \g_n^{t_n}$ such that $\theta_n \to \id$.
\end{lem}

\begin{proof}
Let $\h{\mathcal{A}} \subset \partial \disk$ be the set of all points in $\partial \disk$ at which $\varphi_U^0$ is defined, and let $\mathcal{A} = \{\varphi_U^0(x): x \in \h{\mathcal{A}}\}$.  This set $\mathcal{A}$ is the set of all accessible points in $\partial U$ by Theorem~\ref{lift}.  The set $\mathcal{A}$ is dense in $\partial U$ and the set $\h{\mathcal{A}}$ of lifts of points in $\mathcal{A}$ under $\varphi_U^0$ is dense in $\partial \disk$ by Theorem~\ref{Fatou}.

\begin{claiminproof}
\label{claim:contangle}
For each $t \in [0,1]$, the function $\alpha^t: \partial \disk \to \partial \disk$ which extends the function that maps each $\h{y} \in \h{\mathcal{A}}$ to $\h{y}^t$, and is defined by $\alpha^t(x) = \lim_{\{\h{y} \to  x \,\mid\, \h{y} \in \h{\mathcal{A}} \}} \h{y}^t$ for each $x \in \partial \disk$, is a homeomorphism.  Moreover $\alpha: \partial \disk \times [0,1] \to \partial \disk$, defined by $\alpha(x,t) = \alpha^t(x)$, is an isotopy starting at the identity.
\end{claiminproof}

\begin{proof}[We sketch the proof of \cref{claim:contangle}]
\renewcommand{\qedsymbol}{\textsquare (\cref{claim:contangle})}
Since the restriction $\alpha^t|_{\h{\mathcal{A}}}$ is one-to-one and preserves circular order, it suffices to show that $\alpha^t(\h{\mathcal{A}})$ is dense for each $t$.
 The proof will make use of the following notion:  Let $\mathbb S$ be the unit circle,  $\gamma:\mathbb S\to \mathbb C$ a continuous function and $O$ a point in the unbounded complementary domain of $\gamma(\mathbb S)$. A complementary domain $U$ of $\gamma(\mathbb S)$ is odd if every arc
$J$ from $O $ to a point in the domain intersects $\gamma(\mathbb S)$  an odd number of times; counting with multiplicity  and assuming that every intersection is transverse and the total number of crossings is  finite; cf. \cite[Lemma 2.1]{OT82}.

Fix $\varepsilon>0$. By Theorem~\ref{Riesz}, $\alpha^0(\h{\mathcal{A}})$ is dense. By condition (ii)
of Theorem~\ref{main1technical} and Lemma~\ref{smallup} there exists $\delta>0$ so that for any crosscut $C$ of $X^0$ all lifts of
the paths $\gamma_t^C$ (whose existence follows from condition (ii)) have diameters less than $\varepsilon$. Since $X$ is uniformly perfect one can choose finitely many simple closed curves $S_i$ which bound disjoint closed disks $D_i$ so that $X^0\subset \bigcup_i D_i$,  $X^0\cap \bigcup \partial D_i$ is finite,  and for all $i$ and every component $C$ of $S_i\setminus X^0$, the diameter of $C$ is less than $\delta$. Moreover we can assume that
 for all $t$  \ $\varphi^t_U(0)$ is contained in the unbounded component of $\bigcup \gamma^C_t([0,1])$. Then all lifts $\widehat\gamma_t^C$ have diameter less than $\varepsilon$.
 Let $F^0=\bigcup_i X^0\cap S_i$ and  $F^t=h^t(F^0)$. Since for all $C$ and all $t$, $\gamma_t^C((0,1))\cap X^t=\emptyset$ and $\h{\gamma}_t$ is continuous in the Hausdorff metric, it follows that every point of $X^t\setminus F^t$ is contained in an odd bounded complementary component of $\bigcup \gamma^C_t([0,1])$.

 Every component $C$ of $S_i\setminus X$ is a crosscut
which defines a collection of paths $\gamma_t^C$ by condition (ii)
of Theorem~\ref{main1technical}. For all $t$ let $\h{\mathcal C}_t$ be the collection of all lifts of all the paths $\gamma_t^C$.

Fix $t$.
Suppose that $r$ is a radius of the unit disk $\mathbb D$  so that $R=\varphi^t_U(r)$ lands on a point in $X^t \setminus F^t$.
Then a terminal segment $B$ of $R$  must be in an odd complementary domain  of $\bigcup \gamma^C_t([0,1])$. Let $A =R \setminus B$ be the initial segment of $R$. Then the subsegment $b$ of $r$ that corresponds to $B$ is disjoint from all crosscuts in $\h{\mathcal C}_t$.  Suppose that that $b$ is not contained in the shadow of one of these crosscuts.
 Then we may assume that the intersection of  $a$ and any member of $\h{\mathcal C}_t$ is finite and even. Since we may also assume that the intersection of $a$ with all lifted crosscuts is finite,
the  intersection of $a$ with the union of all members of $\h{\mathcal C}_t$  in a finite even number.
Since $\varphi^t_U$ is a local homeomorphism,  the number of intersections of $A$ with all crosscuts $\gamma_t^C$ is  also even, a contradiction since $A$ terminates in an odd domain.

\end{proof}

Note that by construction, $\mathcal{J}$ is almost a lamination, except that multiple arcs in $\mathcal{J}$ can share the same two endpoints.  In particular, if $C(a_n b_n)$ are circular arcs in $\mathcal{J}$ joining the points $a_n$ and $b_n$ then, after taking a subseuence if necessary, $\lim C(a_n b_n)$ is a circular arc in $\mathcal{J}$ joining $a_\infty$ to $b_\infty$.  From this it follows easily that $\lam^0$ is a lamination and, if $\g_n \in \lam^0$ is the geodesic joining $a_n$ to $b_n$, then $\lim \g_n = \g_\infty$, where $\g_\infty \in \lam^0$ is the geodesic joining $a_\infty$ to $b_\infty$.  Choose lifts $\h{\g}_n^t$ and $\h{\g}_\infty^t$ under $\varphi_U^t$ for each $t \in [0,1]$ as in the proof of \cref{gtarc}, such that $\lim \h{\g}_n^0 = \h{\g}^0_\infty$.

Fix $k$.  By \cref{claim:contangle}, $\lim \h{\g}_n^{t_k} = \h{\g}_\infty^{t_k}$.  This implies immediately that $\liminf \g_n^{t_k} \supset \g_\infty^{t_k}$.  Since the points $a_n$ and $a_\infty$ can be joined by a small crosscut in $U$, it follows from assumption (ii) of \cref{main1technical} that the points $a_n^{t_k}$ and $a_\infty^{t_k}$ can be joined by a small path.  Hence, points $x_n^{t_k}$ in $\g_n^{t_k}$ close to an endpoint (say $a_n^{t_k}$) can be joined to the endpoint $a_n^{t_k}$ by a small path (first by a small arc to a point in $\g_\infty^{t_k}$ and then by a small arc in $U^{t_k}$ to the endpoint $a_\infty^{t_k}$, followed by a small path in $U^{t_k}$ to $a_n^{t_k}$).  By \cref{gehrhay}, the sub-geodesic of $\g_n^{t_k}$ from $x_n^{t_k}$ to $a_n^{t_k}$ is small and we can conclude that $\lim \g_n^{t_k} = \g_\infty^{t_k}$ for each $k$.  Since the maps $\varphi_U^t$ are uniformly convergent on compact subsets, $\liminf \g_\infty^{t_k} \supset \g_\infty^{t_\infty}$.  Since by the above argument the sub-geodesic from a point close to the endpoint of $\g_\infty^{t_k}$ to this endpoint is small, $\lim \g_\infty^{t_k} = \g_\infty^{t_\infty}$.  It is now easy to see that there exist homeomorphisms $\theta_n: \g_\infty^{t_\infty} \to \g_n^{t_n}$ such that $\theta_n \to \id$.
\end{proof}

For each $t \in [0,1]$, we conclude from \cref{gtarc} and \cref{gtconverge} (using $t_n = t$ for all $n$) that $\lam^t$ is a lamination in $U^t$.

\subsection{Proof of \cref{main1technical}}
\label{sec:proof main1}

In this section we will complete the proof of \cref{main1technical} (and hence of \cref{main1} as well).

We will employ here the path midpoint function $\midpt$ described in \cref{thm:midpt} of \cref{sec:midpoints}.

Let $U$ be any bounded complementary domain of $X$, and consider the hyperbolic laminations $\lam^t$ in $U^t$ as constructed above in \cref{sec:laminations}.

Given any element $\g \in \lam^0$, we extend the isotopy $h$ over $\g$ to $h_\g: (X \cup \g) \times [0,1] \to \complex$ by defining $h_\g^t(\midpt(\g)) = \midpt(\g^t)$ and, if $x \in \g$ is located on the subarc with endpoints $\midpt(\g)$ and $a$ (respectively, $b$), then $h_\g^t(x)$ is the unique point on the subarc of $\g^t$ with endpoints $\midpt(\g^t)$ and $a^t$ (respectively, $b^t$) such that $\rho^0(x, \midpt(\g)) = \rho^t(h_\g^t(x), \midpt(\g^t))$, using the hyperbolic metric $\rho^t$ on $U^t$.

Now extend $h$ to $h_\lam: X \cup \bigcup \lam^0 \to \complex$ by defining
\[ h_\lam(x,t) =
\begin{cases}
h(x,t) & \textrm{if $x \in X$} \\
h_\g(x,t) & \textrm{if $x \in \g \in \lam^0$.}
\end{cases}
\]

Then for each $t \in [0,1]$, $h_\lam^t$ is clearly a bijection from $X \cup \bigcup \lam^0$ to $X^t \cup \bigcup \lam^t$.

\begin{claim}
\label{claim:cont}
$h_\lam$ is continuous.
\end{claim}

\begin{proof}[Proof of \cref{claim:cont}]
\renewcommand{\qedsymbol}{\textsquare (\cref{claim:cont})}
Suppose that $(x_i,t_i) \to (x_\infty,t_\infty)$ and $x_i \in \g_i \in \lam^0$.  If there exists $\e > 0$ so that $\diam(\g_i) > \e$ for all $i$, then we may assume, by taking a subsequence if necessary, that $\lim \g_i = \g_\infty \in \lam^0$.  If $x_\infty$ is not an endpoint of $\g_\infty$ then, by uniform convergence of $\varphi_U^t$ on compact sets, $\lim h_\lam(x_i,t_i) = h_\lam(x_\infty,t_\infty)$.  If $x_\infty$ is an endpoint of $\g_\infty$ (so $x_\infty \in X$), then $\rho^0(x_i,\midpt(\g_i)) \to \infty$ and again $\lim h_\lam(x_i,t_i) = h_\lam(x_\infty,t_\infty) = h(x_\infty,t_\infty)$.  Hence we may assume that $\lim \diam(\g_i) = 0$.  Then $x_\infty \in X$ and $\lim \diam(h_\lam^{t_i}(\g_i)) = 0$.  Hence, if $a_i$ is an endpoint of $\g_i$, then $\lim h_\lam(x_i,t_i) = \lim h(a_i,t_i) = h(x_\infty,t_\infty)$ as desired.
\end{proof}

Finally, we repeat the above procedure on each bounded complementary domain $U$ of $X$ to extend $h$ over the hyperbolic lamination obtained from the Kulkarni-Pinkall construction as in \cref{sec:laminations} on each such $U$.  The result is a function $H: Y \times [0,1] \to \complex$ which is defined on the union $Y$ of $X$ with all the hyperbolic laminations of all bounded complementary domains of $X$.  Note that for any $\e > 0$, there are only finitely many bounded complementary domains of $X$ which contain a disk of diameter at least $\e$, and hence there are only finitely many such domains whose corresponding hyperbolic lamination contains an arc of diameter at least $\e$.  This implies, as above, that $H$ is continuous.

Note that each bounded complementary domain of $Y$ is a gap of the hyperbolic lamination of one of the bounded complementary domains of $X$.  Since all such gaps are simply connected, $Y$ is a continuum.  Hence by \cite{ot10} the isotopy $H$ of $Y$ can be extended over the entire plane.

This completes the proof of \cref{main1technical}.  By the comments at the end of \cref{sec:lifts moving}, this also completes the proof of \cref{main1}.

\bigskip

In \cref{main1} we assumed that $X^t$ is uniformly perfect for each $t \in [0,1]$.  This assumption allows for the use of the powerful analytic results described in \cref{sec:analytic covering maps}.  It is natural to wonder if this assumption is really needed.  We conjecture that this is not the case.

\begin{conj}
Suppose that $X$ is a plane compactum and $h: X \times [0,1] \to \complex$ is an isotopy starting at the identity.  Then the following are equivalent:
\begin{enumerate}
\item $h$ extends to an isotopy of the entire plane,
\item for each $\e > 0$ there exists $\delta > 0$ such that for every complementary domain $U$ of $X$ and each crosscut $Q$ of $U$ with $\diam(Q) < \delta$, $h$ can be extended to an isotopy $h_Q: (X \cup Q) \times [0,1] \to \complex$ such that for all $t \in [0,1]$, $\diam(h^t(Q)) < \e$.
\end{enumerate}
\end{conj}

\section{Compact sets with large components}
\label{sec:large components}

The remaining part of this paper is devoted to a proof of the following theorem.

\begin{thm}
\label{main2}
Suppose $X \subset \complex$ is a compact set for which there exists $\eta > 0$ such that every component of $X$ has diameter bigger than $\eta$.  Let $h: X \times [0,1] \to \complex$ be an isotopy which starts at the identity.  Then $h$ extends to an isotopy of the entire plane which starts at the identity.
\end{thm}

Suppose $X \subset \complex$ is a compact set for which there exists $\eta > 0$ such that every component of $X$ has diameter bigger than $\eta$.  Let $h: X \times [0,1] \to \complex$ be an isotopy which starts at the identity.

Clearly in this case $X^t$ is uniformly perfect with the same constant $k$ for each $t \in [0,1]$, and we may assume that $X$ is encircled.  By scaling, we may also assume that for any $a \in X$ and any component $C$ of $X$, there exists $c \in C$ such that $|a^t - c^t| \geq 1$ for all $t \in [0,1]$.  \emph{We will make these assumptions for the remainder of the paper}.

We will prove \cref{main2} using the characterization from \cref{main1technical}.  To this end, we fix (again for the remainder of the paper) an arbitrary bounded complementary domain $U$ of $X$.

To satisfy condition (ii) of \cref{main1technical} we must construct, for a sufficiently small crosscut $Q$ of $U$ with endpoints $a$ and $b$, a family of paths $\gamma_t$ in $U^t$ with endpoints $a^t$ and $b^t$, which remain small during the isotopy, such that $\gamma_0$ is homotopic to $Q$ in $U$ with endpoints fixed, and which can be lifted under $\varphi_U^t$ to paths $\h{\gamma}_t$ in $\disk$ which are continuous in the Hausdorff metric.  We will show first that, in the case that $X$ has large components, it suffices to construct the family of paths $\gamma_t$ to be continuous in the Hausdorff metric.

\begin{lem}
\label{liftexist}
Let $a,b \in \partial U$.  Suppose that $\{\gamma_t: t \in [0,1]\}$ is a family such that $\gamma_t$ is a path in $U^t$ joining $a^t$ and $b^t$ with $\diam(\gamma_t([0,1])) < \frac{1}{2}$ for each $t \in [0,1]$, and the sets $\gamma_t([0,1])$ vary continuously in $t$ with respect to the Hausdorff metric.  Then there are lifts $\h{\gamma}_t$ of the paths $\gamma_t$ under $\varphi_U^t$ such that the sets $\h{\gamma}_t([0,1])$ also vary continuously in $t$ with respect to the Hausdorff metric.
\end{lem}

\begin{proof}
Suppose that the family $\gamma_t$ is as specified in the statement.  Recall that $d_H$ denotes the Hausdorff distance.  Fix $t_0 \in [0,1]$.  It suffices to show that, given a lift $\h{\gamma}_{t_0}$ of $\gamma_{t_0}$ and $0 < \e < \frac{1}{2}$, there exists $\delta > 0$ and lifts $\h{\gamma}_t$ of $\gamma_t$ for $|t - t_0| < \delta$ such that $d_H(\h{\gamma}_t([0,1]), \h{\gamma}_{t_0}([0,1])) < \e$.

By \cref{smallup} we can choose small disjoint open balls $B_a$ centered at $a^{t_0}$ and $B_b$ centered at $b^{t_0}$ of diameters less than $\frac{1}{4}$ such that for all $t$ and any component $C$ of $\partial B_a \cap U^t$ or $\partial B_b \cap U^t$, the diameter of each component of $(\varphi_U^t)^{-1}(C)$ is less than $\frac{\e}{4}$.

Let $s_a,s_b \in (0,1)$ be the numbers such that $\gamma_{t_0}(s_a) \in \partial B_a$, $\gamma_{t_0}([0,s_a)) \subset B_a$, $\gamma_{t_0}(s_b) \in \partial B_b$, and $\gamma_{t_0}((s_b,1]) \subset B_b$.  Denote $z_a = \gamma_{t_0}(s_a)$ and $z_b = \gamma_{t_0}(s_b)$.  Choose an open set $O \subset \complex$ such that $\gamma_{t_0}([s_a,s_b]) \subset O$, $\overline{O} \subset U^{t_0}$, and the diameter of $O \cup B_a \cup B_b$ is less than $1$.  For $t$ sufficiently close to $t_0$, we have $\overline{O} \subset U^t$ and $\gamma_t([0,1]) \subset O \cup B_a \cup B_b$.  Since each component of $X^t$ has diameter greater than $1$, we have that no bounded complementary component of $O \cup (B_a \cup B_b \sm X^t)$ contains any points of $X^t$.  It follows that there exists a simply connected open set $P_t$ in $U^t$ such that $\gamma_t((0,1)) \cup O \subset P_t$.  This means that the covering map $\varphi_U^t$ maps each component of $(\varphi_U^t)^{-1}(P_t)$ homeomorphically onto $P_t$.

Since the maps $\varphi_U^t$ converge uniformly on compact sets as $t \to t_0$, for $t$ sufficiently close to $t_0$ there exists exactly one component $\h{P}_t$ of $(\varphi_U^t)^{-1}(P_t)$ such that $\h{\gamma}_{t_0}([s_a,s_b]) \subset \h{P}_t$.  For such $t$, define the lift $\h{\gamma}_t$ of $\gamma_t$ by $\h{\gamma}_t = (\varphi_U^t |_{\h{P}_t})^{-1} \circ \gamma_t$.

To see that these lifts are Hausdorff close to $\h{\gamma}_{t_0}$, let $\delta > 0$ be small enough so that for all $t$ with $|t - t_0| < \delta$ we have:

\begin{enumerate}
\item There exists $\nu > 0$ such that $|(\varphi_U^t|_{\h{P}_t})^{-1}(x_1) - (\varphi_U^{t_0}|_{\h{P}_{t_0}})^{-1}(x_2)| < \frac{\e}{2}$ for all $x_1,x_2 \in \complex$ with $|x_1 - x_2| < \nu$ and either $x_1 \in O$ or $x_2 \in O$;
\item $d_H(\gamma_t([0,1]), \gamma_{t_0}([0,1])) < \nu$; and
\item $\gamma_t([0,1]) \cap (\partial B_a \sm O) = \0$ and $\gamma_t([0,1]) \cap (\partial B_b \sm O) = \0$.
\end{enumerate}

Given $t$ with $|t - t_0| < \delta$, let $C_{a,t}$ be the component of $\partial B_a \sm X^t$ which contains $z_a$, and let $C_{b,t}$ be the component of $\partial B_b \sm X^t$ which contains $z_b$.  Let $\h{C}_{a,t}$ and $\h{C}_{b,t}$ be lifts of $C_{a,t}$ and $C_{b,t}$ which contain $(\varphi_U^t |_{\h{P}_t})^{-1}(z_a)$ and $(\varphi_U^t |_{\h{P}_t})^{-1}(z_b)$, respectively.  By the choice of $B_a$ and $B_b$, the diameters of $\h{C}_{a,t}$ and $\h{C}_{t,b}$ are less than $\frac{\e}{4}$.  It follows from (iii) that $\h{\gamma}_t([0,1])$ is contained in $(\varphi_U^t|_{\h{P}_t})^{-1}(O)$ together with the small region  under $\h{C}_{t,a}$ and the small region
 under $\h{C}_{t,b}$.  Note that these small regions have diameters less than $\frac{\e}{2}$.  This means that for every point $\h{p}$ in $\h{\gamma}_t([0,1])$ there is a point $\h{q} \in \h{\gamma}_t([0,1]) \cap (\varphi_U^t|_{\h{P}_t})^{-1}(O)$ such that $|\h{p} - \h{q}| < \frac{\e}{2}$.  Then, since $q = \varphi_U^t(\h{q}) \in O$, by (ii) there is a point $r \in \gamma_{t_0}([0,1])$ such that $|q - r| < \nu$.  If we let $\h{r}$ be the lift $\h{r} = (\varphi_U^{t_0}|_{\h{P}_{t_0}})^{-1}(r) \in \h{\gamma}_{t_0}([0,1])$, then by (i) we have $|\h{q} - \h{r}| < \frac{\e}{2}$.  Then by the triangle inequality, $|\h{p} - \h{r}| < \e$.  Similarly, we can show that for any $\h{r} \in \h{\gamma}_{t_0}([0,1])$ there is a point $\h{p} \in \h{\gamma}_t([0,1])$ with $|\h{p} - \h{r}| < \e$.  Thus $d_H(\h{\gamma}_t([0,1]), \h{\gamma}_{t_0}([0,1])) < \e$.
\end{proof}

\begin{notation}[$\e$, $\nu$]
For the remainder of the paper, we fix an arbitrary $\e > 0$.  For later use, fix $0 < \nu < \frac{1}{3}$ small enough so that $\frac{8 \nu}{1 - \nu} < \frac{\e}{2}$.
\end{notation}

To prove \cref{main2}, it remains to show that there exists $\delta > 0$ such that if $Q$ is a crosscut of $U$ with endpoints $a$ and $b$ with diameter less than $\delta$, there is a family of paths $\gamma_t$ such that (1) $\gamma_t$ is a path in $U^t$ joining $a^t$ and $b^t$ for each $t \in [0,1]$, (2) $\gamma_0$ is homotopic to $Q$ in $U$ with endpoints fixed, (3) $\diam(\gamma_t([0,1])) < \e$ for all $t \in [0,1]$, and (4) the sets $\gamma_t([0,1])$ vary continuously in $t$ with respect to the Hausdorff metric.

In Section 4.1, we will transform the compactum $X$, so that the crosscut $Q$ becomes the straight line segment $[0,1]$ in the plane, to simplify the ensuing constructions and arguments.  We will refer to the transformed plane as the ``normalized plane'', and the image of $X$ will be denoted by $\til{X}$.  In Section 4.2, we will lift the isotopy under an exponential covering map.  The domain of the covering map will be called the ``exponential plane'', and the preimage of $\til{X}$ will be denoted by $\bol{X}$.  In Sections 4.3 and 4.4 we will replace the lift of the crosscut $[0,1]$ of $\til{X}$ by an equidistant set which varies continuously in $t$.  The projection of this equidistant set to the original plane containing $X^t$ will be shown in Section 4.5 to be the desired path $\gamma_t$.

\subsection{The normalized plane}
\label{sec:norm plane}

In the following sections, we will make use of a covering map (which we will refer to as the ``exponential map'') of the plane minus the endpoints of a crosscut $Q$.  In order to simplify the notation and work with a single exponential map below we will normalize the compactum $X$ and the crosscut $Q$ of $X$ with end points $a$ and $b$ so that for all $t$, $a^t = 0$, $b^t = 1$, and $Q$ becomes the straight line segment $(0,1) \subset \real$.

By composing with translations it is easy to see that given a crosscut $Q$ of $X$ with endpoints $a$ and $b$ we can always assume that the point $a$ is the origin $0$ and that this point remains fixed throughout the isotopy (i.e., $a^t = 0$ for all $t$).

Let $Q$ be a crosscut of $U$ with endpoints $0$ and $b$ such that $\diam(Q) < \frac{1}{4}$.  We will impose further restrictions on the diameter of $Q$ later.

Since all arcs in the plane are tame, there exists a homeomorphism $\Theta: \complex \to \complex$ such that $\Theta(Q)$ is the straight line segment joining the points $0$ and $b$, $\Theta(0) = 0$, $\Theta(b) = b$ and $\Theta|_{\complex \sm B(0,2\diam(Q))} = \id_{\complex \sm B(0,2\diam(Q))}$.  Let $L^t: \complex \to \complex$ be the linear map of the complex plane defined by $L^t(z) = \frac{1}{\Theta(b^t)} \,z$.

\begin{notation}[$\til{X}$, $\til{x}^t$]
Define $\til{X} = L^0 \circ \Theta(X)$ and define the isotopy
$\til{h}: \til{X} \times [0,1] \to \complex$ by
\[ \til{h}(\til{x},t) = L^t \circ \Theta \circ h((L^0 \circ \Theta)^{-1}(\til{x}),t) = L^t \circ \Theta(x^t) .\]

Here and below we adopt the notation that $\til{x} = L^0 \circ \Theta(x)$ for all $x \in X$ and, hence, $\til{h}^t(\til{x}) = \til{x}^t = L^t \circ \Theta(x^t)$.  As indicated above, we will use ordinary letters to denote objects in the plane containing $X$ and attach a tilde to the corresponding objects in the normalized plane (the plane containing $\til{X}$).
\end{notation}

In the next lemma we establish some simple properties of the induced isotopy $\til{h}$.

\begin{lem}
\label{sizes}
There exists $\delta > 0$ such that if the crosscut $Q$ of $X$ with endpoints $0$ and $b$ has diameter $\diam(Q) < \delta$, then the induced isotopy $\til{h}: \til{X} \times [0,1] \to \complex$ has the following properties:
\begin{enumerate}
\item $\til{h}^0 = \id_{\til{X}}$, $\til{X}$ contains the points $0$ and $1$, the isotopy $\til{h}$ fixes these points and the segment $(0,1) \subset \real$ in the complex plane is disjoint from $\til{X}$;
\item If $\til{x}^s \in (0,1)$ for some $s \in [0,1]$, then for each $t \in [0,1]$, $|\til{x}^t| < \frac{\nu}{|\Theta(b^t)|}$; and
\item For every component $\til{C}$ of $\til{X}$ there exists a point $\til{c} \in \til{C}$ such that for all $t \in [0,1]$, $|\til{c}^t| \geq \frac{1}{|\Theta(b^t)|}$.
\end{enumerate}
\end{lem}

\begin{proof}
It follows immediately that $\til{h}^0 = \id|_{\til{X}}$, the isotopy $\til{h}$ fixes the points $0$ and $1$ and that the interval $(0,1)$ is disjoint from $\til{X}$.  Hence (i) holds.

Since $h$ is uniformly continuous we can choose $0 < \delta < \frac{\nu}{4}$ so that if $x \in X$ and $|x^s| < 2\delta$ for some $s \in [0,1]$, then $|x^t| < \frac{\nu}{2}$ for all $t$.  Suppose $\til{x}^s \in (0,1)$ for some $s\in [0,1]$, then $x^s \in Q$ and hence $|x^t| < \frac{\nu}{2}$ for all $t$.  Then $|\til{x}\,^t| < \frac{\nu}{2|\Theta(b^t)|} + \frac{2\delta}{|\Theta(b^t)|} \leq \frac{\nu}{|\Theta(b^t)|}$ using that $\Theta|_{B\setminus B(0,2\delta)}=id$ and so (ii) holds.

By the standing assumption on $X$ stated after \cref{main2}, for every component $C$ of $X$ there exists a point $c \in C$ such that for all $t$, $|c^t| > 1$.  Note that $\Theta(c^t) = c^t$ for all $t$. Hence, $|\til{c}\,^t| \geq \frac{|c^t|}{\Theta(b^t)|} \geq \frac{1}{|\Theta(b^t)|}$ for all $t$ and (iii) holds.
\end{proof}

\subsection{The exponential plane}
\label{sec:exp plane}

Define the covering map
\[ \wE: \; \complex \sm \{(2n+1)\pi i: n \in \mathbb{Z}\} \;\to\; \complex \sm \{0,1\} \]
by
\[ \wE(z) = \frac{e^z}{e^z + 1} .\]

The function $\wE$ is periodic with period $2\pi i$, and satisfies
\[ \lim_{\real(z) \rightarrow \infty} \wE(z) = 1, \quad \lim_{\real(z) \rightarrow -\infty} \wE(z) = 0, \quad \wE(\real) = (0,1) ,\]
and has poles at each point $(2n+1)\pi i$, $n \in \mathbb{Z}$.

Note that $\wE$ is the composition of the maps $e^z$ and the M\"{o}bius transformation $f(w) = \frac{w}{w+1}$.  Hence the vertical line through a point $x \in \real$ is first mapped (by the covering map $e^z$) to the circle with center $0$ and radius $e^x$ and, if $x \neq 0$, then mapped by $f$ to the circle with center $\frac{e^{2x}}{e^{2x} - 1}$ and radius $\left| \frac{e^{x}}{e^{2x} - 1} \right|$.  The imaginary axis is mapped to the vertical line through the point $x = \frac{1}{2}$ with the points at the poles $(2n+1)\pi i$ mapped to infinity.

\begin{notation}[$\bol{X}$, $\bol{x}^t$, $\bol{E}_n(r)$]
Denote by boldface $\bol{X}$ the preimage of $\til{X}$ under the covering map $\wE$, and in general we will use boldface letters to represent points and subsets of the exponential plane (the plane containing $\bol{X}$).

The isotopy $\til{h}$ of $\til{X}$ lifts to an isotopy $\bol{h}$ of $\bol{X}$; that is, $\bol{h}: \bol{X} \times [0,1] \to \complex$ is the map satisfying $\bol{h}^0 = \id_{\bol{X}}$ and $\wE(\bol{h}(\bol{x}, t)) = \til{h}(\wE(\bol{x}), t)$ for every $\bol{x} \in \bol{X}$ and all $t \in [0,1]$.  As above, given a point $\bol{x} \in \bol{X}$ (a subset $\bol{A} \subseteq \bol{X}$) and $t \in [0,1]$, denote $\bol{x}^t = \bol{h}(\bol{x}, t)$ (respectively, $\bol{A}^t = \bol{h}(\bol{A}, t)$).

For each $n \in \mathbb{Z}$ and each $r > 0$, let $\bol{E}_n(r) = B((2n+1) \pi i, r)$ be the ball of radius $r$ centered at the point $(2n+1) \pi i$.
\end{notation}

\begin{lem}
\label{ball-like}
There exists $0 < K < \pi$ such that for any $0 < r \leq K$,
\begin{enumerate}
\itemsep=5pt
\item $\displaystyle \wE \left( \bigcup_{n \in \mathbb{Z}} \bol{E}_n(r) \right) \subset \complex \sm B \left( 0,\frac{1}{2r} \right)$;
\item $\displaystyle \wE \left( \complex \sm \bigcup_{n \in \mathbb{Z}} \bol{E}_n(r) \right) \subset B \left( 0,\frac{2}{r} \right)$.
\end{enumerate}
\end{lem}

\begin{proof}
For any $n \in \mathbb{Z}$ and sufficiently small $|z|$, we have
\[ e^{(2n+1)\pi i + z} = -e^z \approx -1 - z \]
and hence $\wE((2n+1)\pi i + z) \approx \frac{1 + z}{z}$.  In particular, there exists $0 < K < \pi$ such that for all $|z| \leq K$
\[ \frac{1}{2|z|} \leq |\wE((2n+1)\pi i + z)| \leq \frac{2}{|z|} .\]
Let $S_n = \partial B((2n+1)\pi i, r)$.  Then, by the above inequality, $T = \wE(\bigcup_n S_n)$ is an essential simple closed curve in the annulus centered around the origin $0$ with inner radius $\frac{1}{2r}$ and outer radius $\frac{2}{r}$.  Since $\wE$ is periodic, all $S_n$ have the same image $T$, and $\wE^{-1}(T) = \bigcup_n S_n$.  It follows that $\wE(\bigcup_n B((2n+1)\pi i, r) \sm (2n+1)\pi i)$ is contained in the unbounded complementary domain of $T$ and $\wE(\complex \sm \bigcup_n B((2n+1)\pi i, r))$ is contained in the bounded complementary domain of $T$.  Hence, $\wE(\bigcup_n B((2n+1)\pi i, r)) \subset \complex \sm B(0,\frac{1}{2r})$ and $\wE(\complex \sm \bigcup_n B((2n+1)\pi i, r)) \subset B(0,\frac{2}{r})$.
\end{proof}

\subsection{Components of $\bol{X}^t$}
\label{sec:exp components}

We say a component $\bol{C}$ of $\bol{X}^t$ ($t \in [0,1]$) is \emph{unbounded to the right} (respectively \emph{left}) if $\proj_\real(\bol{C}) \subseteq \real$ is not bounded from above (respectively from below).

For convenience we denote the horizontal strip $\{x + iy \in \complex: x \in \real,\; 2n\pi < y < 2(n+1)\pi\}$ simply by $\bol{HS}_n$.  Observe that since $\til{X} \cap (0,1) = \emptyset$ and $\wE^{-1}((0,1)) = \bigcup_{n \in \mathbb{Z}} \{x + iy \in \complex: x \in \real,\; y = 2n\pi\}$, we have that $\bol{X} \subset \bigcup_{n \in \mathbb{Z}} \bol{HS}_n$.

\begin{lem}
\label{uniqueB_n}
There exists $\delta > 0$ such that if the crosscut $Q$ of $X$ with endpoints $0$ and $b$ has diameter $\diam(Q) < \delta$, then the following holds for the induced isotopy $\bol{h}$ of $\bol{X}$:

Given a component $\bol{C}$ of $\bol{X}$, let $n \in \mathbb{Z}$ be such that $\bol{C}$ is contained in the horizontal strip $\bol{HS}_n$.  Let $\til{D}$ be the component of $\til{X}$ that contains $\wE(\bol{C})$.  Then:
\begin{enumerate}
\item if $\til{D} \cap \{0,1\} = \0$, then $\bol{C}^t \cap \bol{E}_n \left( \frac{|\Theta(b^t)|}{2} \right) \neq \0$ for all $t \in [0,1]$;
    \item $\bol{C}^t \cap \bol{E}_m \left( \frac{|\Theta(b^t)|}{2\nu} \right) = \0$ for all $m \neq n$ and all $t \in [0,1]$; and
\item if $\til{D} \cap \{0,1\} \neq \0$, then $\bol{C}$ is unbounded to the left, to the right, or both.
\end{enumerate}

Furthermore, there exist for each $k \in \mathbb{Z}$ components $\bol{L}_k$ and $\bol{R}_k$ of $\bol{X} \cap \bol{HS}_k$ such that for all $t \in [0,1]$, $\bol{L}_k^t$ is unbounded to the left and $\bol{R}_k^t$ is unbounded to the right.  Moreover, these may be chosen so that either $\bol{L}_k^t \cap \bol{E}_k \left( \frac{|\Theta(b^t)|}{2} \right) \neq \0$ for all $k \in \mathbb{Z}$ or $\bol{R}_k^t \cap \bol{E}_k \left( \frac{|\Theta(b^t)|}{2} \right) \neq \0$ for all $k \in \mathbb{Z}$.
\end{lem}

\begin{figure}
\begin{center}

\begin{subfigure}{}
\includegraphics{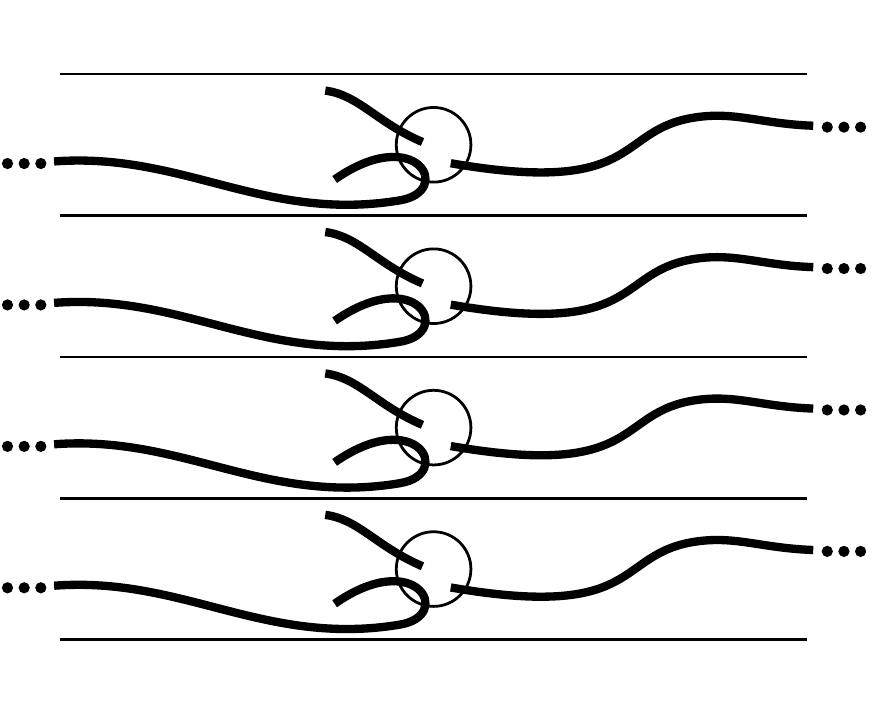}
\end{subfigure}

\vspace{0.15in}

\begin{subfigure}{}
\includegraphics{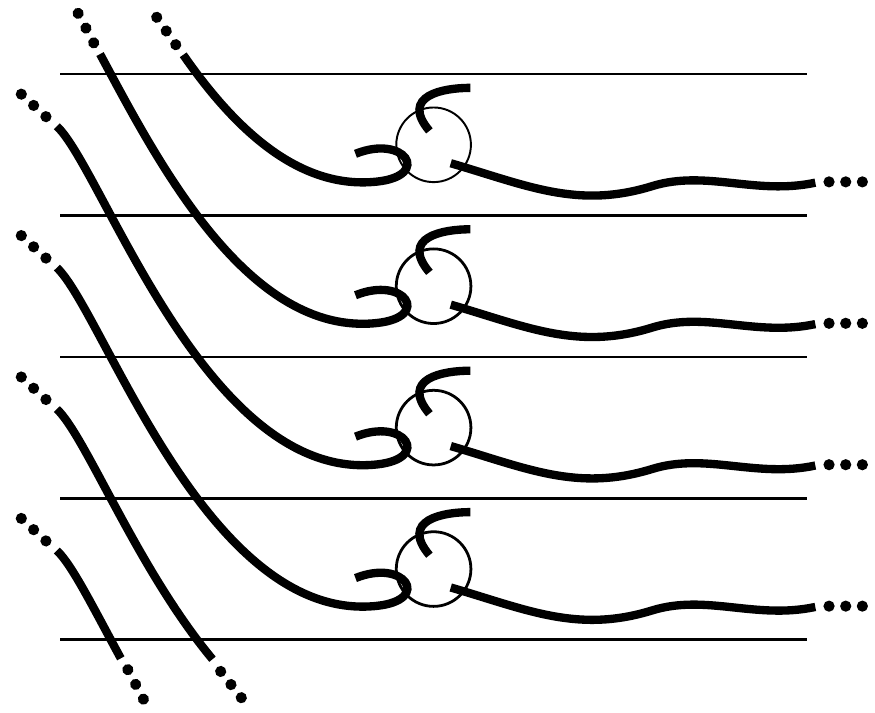}
\end{subfigure}

\end{center}

\caption{An illustration of an example of the set $\bol{X}^t$ at $t = 0$ (above) and at a later moment $t > 0$ (below).  The horizontal lines are the preimages of $(0,1)$ under $\wE$, and the balls depicted are the sets $\bol{E}_n \left( \frac{|\Theta(b^t)|}{2} \right)$.}
\label{fig:exp plane}
\end{figure}

\begin{proof}
Adopt the notation introduced in the Lemma and assume $\bol{C}$ is contained in the horizontal strip $\bol{HS}_n$.  Let $0 < K < \pi$ be as in \cref{ball-like}.  Choose $\delta > 0$ so small that $\frac{|\Theta(b^t)|}{\nu} < K$ for all $t$.

Suppose that $\til{D} \cap \{0,1\} = \0$.  Then $\wE(\bol{C}) = \til{D}$.  By \cref{sizes}(iii), we can choose $\til{c} \in \til{D}$ such that $|\til{c}^t| \geq \frac{1}{|\Theta(b^t)|}$ for all $t$.  By \cref{ball-like}(ii), $\wE \left( \complex \sm \bol{E}_n \left( \frac{|\Theta(b^t)|}{2} \right) \right) \subset B(0,\frac{1}{|\Theta(b^t)|})$.  Hence we can choose $\bol{c}^0 \in \bol{E}_n \left( \frac{|\Theta(b^0)|}{2} \right) \cap \bol{C}$ such that $\wE(\bol{c}^0) = \til{c}^0$, and then $\bol{c}^t \in \bol{C}^t \cap \bol{E}_n \left( \frac{|\Theta(b^t)|}{2} \right)$ for all $t$.  This completes the proof of (i).

Note that for all $n \in \mathbb{Z}$,  $\wE(\real \times \{2n\pi i\}) = (0,1) \subset \real$ and, hence, $\bol{X} \cap (\real \times \{2n\pi i\}) = \0$ for all $n \in \mathbb{Z}$.  To see that $\bol{C}^t \cap \bol{E}_m \left( \frac{|\Theta(b^t)|}{2\nu} \right) = \0$ for $m \neq n$ and all $t$, note first that this is the case at $t = 0$ since $\bol{C}^0 = \bol{C} \subset \bol{HS}_n$.  In order for a point $\bol{x}^s \in \bol{C}^s$ to enter a ball $\bol{E}_m \left( \frac{|\Theta(b^s)|}{2\nu} \right)$ with $n \neq m$ for some $s > 0$, it would first have to cross one of the horizontal boundary lines of $\bol{HS}_n$, say $\bol{x}^u \in \real \times \{2n\pi i\}$ for some $0 < u < s$.  Then $\wE(\bol{x}^u) = \til{x}^u \in (0,1) \subset \real$.  Hence by \cref{sizes}(ii), $|\til{x}\,^t| < \frac{\nu}{|\Theta(b^t)|}$ for all $t$.  Since by \cref{ball-like}(i), $\wE \left( \bol{E}_m \left( \frac{|\Theta(b^t)|}{2\nu} \right) \right) \subset \complex \sm B \left( 0, \frac{\nu}{|\Theta(b^t)|} \right)$ for all $t$, $\bol{x}^s \notin \bol{E}_m \left( \frac{|\Theta(b^s)|}{2\nu} \right)$, a contradiction.  This completes the proof of (ii).

Suppose next that $\til{D} \cap \{0,1\} \neq \0$.  Then $\wE(\bol{C}) = \til{C}$ is a component of $\til{D} \sm \{0,1\}$ such that $\overline{\til{C}} \cap \{0,1\} \neq \0$.  Hence $\bol{C}$ is unbounded to the left or to the right (or both).  This completes the proof of (iii).

There must exist components $\til{L}$ and $\til{R}$ of $\til{X} \sm \{0,1\}$ such that $0$ is in the closure of $\til{L}$ and $1$ is in the closure of $\til{R}$.  For each $k \in \mathbb{Z}$, let $\bol{L}_k$ be the lift of $\til{L}$ under $\wE$ which is contained in the strip $\bol{HS}_k$, and similarly define $\bol{R}_k$.  Then since the closure of $\til{L}\,^t$ contains $0$ and the closure of $\til{R}\,^t$ contains $1$ for all $t \in [0,1]$, we have that for each $k \in \mathbb{Z}$, the lift $\bol{L}_k^t$ is unbounded to the left and the lift $\bol{R}_k^t$ is unbounded to the right for all $t \in [0,1]$.

Finally, by \cref{sizes}(iii), there exists a component $\til{S}$ of $\til{X} \sm \{0,1\}$ whose closure contains $0$ or $1$, which contains a point $\til{c} \in \til{S}$ such that $|\til{c}\,^t| \geq \frac{1}{|\Theta(b^t)|}$.  Then, as in the proof of (ii), the component $\bol{S}_k^t$ of $\wE^{-1}(\til{S}^t)$ which contains the lift $\bol{c}^t_k \in \bol{HS}_k$ of $\til{c}$ under $\wE$ is unbounded to the left or to the right for all $k$ and $t$ and intersects $\bol{E}_k \left( \frac{|\Theta(b^t)|}{2} \right)$ as required.
\end{proof}

\begin{notation}[$\bol{A},\bol{B}$, $\mathfrak{A},\mathfrak{B}$]
Let $\bol{A}$ denote the set of all points of $\bol{X}$ above $\real$ and $\bol{B}$ the set of all points of $\bol{X}$ below $\real$.  Recall that $\bol{X} \cap \real = \0$, so $\bol{X} = \bol{A} \cup \bol{B}$.  For each $t \in [0,1]$, let
\[ \mathfrak{A}^t = \bol{A}^t \cup \bigcup_{n \geq 0} \overline{\bol{E}_n \left( \frac{|\Theta(b^t)|}{2} \right) } \quad\textrm{and}\quad \mathfrak{B}^t = \bol{B}^t \cup \bigcup_{n < 0} \overline{\bol{E}_n \left( \frac{|\Theta(b^t)|}{2} \right) } .\]
Then $\mathfrak{A}^t$ and $\mathfrak{B}^t$ are disjoint closed sets, and by \cref{uniqueB_n}, each component of $\mathfrak{A}^t$ and of $\mathfrak{B}^t$ is either unbounded to the left or to the right.
\end{notation}

\begin{lem}
\label{lem:vertical strip}
For each $r > 0$, there exists a lower bound $\ell \in \real$ (respectively upper bound $u \in \real$) such that for all $t \in [0,1]$, if $c + di \in \bol{A}^t$ (respectively $\bol{B}^t$) and $|c| \leq r$, then $d \geq \ell$ (respectively $d \leq u$).
\end{lem}

\begin{proof}
Let $\imag$ denote the imaginary axis, so that $[-r,r] \times \imag$ is the strip in the plane between the vertical lines through $r$ and $-r$.

By uniform continuity of $\til{h}$ and the fact that $\til{h}$ leaves $0$ and $1$ fixed, there must exist for each $r > 0$ an $r' > r$ such that for all $\bol{x} \in \bol{X}$, if $\bol{x}^s \in ((-\infty,-r'] \cup [r',\infty)) \times \imag$ for some $s \in [0,1]$ then for all $t \in [0,1]$, $\bol{x}^t \notin [-r,r] \times \imag$.

Given a point $\bol{x} \in \bol{A} \cap ([-r',r'] \times \imag)$, let $\til{x} = \wE(\bol{x})$ be the corresponding point of $\til{X}$.  Every time $\bol{x}$ travels vertically within the strip $[-r',r'] \times \imag$ a distance $2\pi$, the point $\til{x}$ travels around a disk of fixed radius (depending on $r'$) centered at $0$ or at $1$.  By uniform continuity and compactness of $X$, this can only happen a uniformly bounded number of times.  The result follows.
\end{proof}

\begin{cor}
\label{compact in strip}
Let $\bol{C}$ is any component of $\bol{X}$.  Then for any $r > 0$ and any $t \in [0,1]$, the set $\bol{C}^t \cap \{x + yi: x \in [-r,r]\}$ is compact.
\end{cor}

\begin{proof}
Because the set $\bol{X}^t$ is periodic with period $2\pi i$, there exists an integer $k$ such that if $\bol{D}$ is the copy of $\bol{C}$ shifted vertically by $2\pi k$, then without loss of generality $\bol{C} \subset \bol{A}$ and $\bol{D} \subset \bol{B}$.  Then by \cref{lem:vertical strip}, $\bol{C}^t$ is bounded below in the strip $\{x + yi: x \in [-r,r]\}$, and $\bol{D}^t$ is bounded above in this strip.  By periodicity, it follows that $\bol{C}^t$ is also bounded above in this strip.
\end{proof}

\begin{defn}
Given two distinct components $\bol{C},\bol{D}$ of $\bol{X}$ which are both unbounded to the right (respectively, to the left), we say that $\bol{C}$ \emph{lies above} $\bol{D}$ if there is some $R > 0$ such that for all $x \in \real$ with $x \geq R$ (respectively, $x \leq -R$), $\max(y \in \real: x + iy \in \bol{C}) > \max(y \in \real: x + iy \in \bol{D})$ and also $\min(y \in \real: x + iy \in \bol{C}) > \min(y \in \real: x + iy \in \bol{D})$.
\end{defn}

Note that it follows immediately from the definition  of $\mathfrak{A}^0$ and $\mathfrak{B}^0$ that if $\bol{C}$ and $\bol{D}$ are components of $\mathfrak{A}^0$ and $\mathfrak{B}^0$, respectively, which are unbounded on the same side, then $\bol{C}$ lies above $\bol{D}$.  The following Lemma follows from this fact.  The proof, which is left to the reader, is very similar to the proof of Lemma 2.5 in \cite{ot10}.

\begin{lem}
\label{lem:dichotomy stable}
There exists $\delta > 0$ such that if the crosscut $Q$ of $X$ with endpoints $0$ and $b$ has diameter $\diam(Q) < \delta$, then the following holds for the induced isotopy $\bol{h}$ of $\bol{X}$:

Let $\bol{C}$ and $\bol{D}$ be components of $\mathfrak{A}^0$ and $\mathfrak{B}^0$, respectively,  which are both unbounded to the same side. Then $\bol{C}^t$ lies above $\bol{D}^t$ for all $t \in [0,1]$.

Consequently, if $\bol{E}$ and $\bol{F}$ are components of $\bol{A}$ and $\bol{B}$, respectively, which are both unbounded to the same side, then $\bol{E}^t$ lies above $\bol{F}^t$ for all $t \in [0,1]$.
\end{lem}

\subsection{Equidistant set between $\bol{A}^t$ and $\bol{B}^t$}
\label{sec:equi path}

For the remainder of this section, we assume that $\delta > 0$ is chosen so that the conclusions of \cref{uniqueB_n} and \cref{lem:dichotomy stable} hold.  We also assume that the crosscut $Q$ has diameter less than $\delta$.

Recall that disjoint closed sets $A_1$ and $A_2$ in $\complex$  are \emph{non-interlaced} if whenever $B(c,r)$ is an open disk contained in the complement of $A_1 \cup A_2$, there are disjoint arcs $C_1,C_2 \subset \partial B(c,r)$ such that $A_1 \cap \partial B(c,r) \subset C_1$ and $A_2 \cap \partial B(c,r) \subset C_2$.  We allow for the possibility that $C_1 = \0$ in the case that $A_2 \cap \partial B(c,r) = \partial B(c,r)$, and vice versa.

\begin{lem}
\label{lem:noninterlaced}
$\bol{A}^t$ and $\bol{B}^t$ are non-interlaced for all $t \in [0,1]$.
\end{lem}

\begin{proof}
Fix $t \in [0,1]$.  Let $B \subset \complex \sm (\bol{A}^t \cup \bol{B}^t)$ be a round open ball, and suppose for a contradiction that there exist points $\bol{a}_1,\bol{a}_2 \in \partial B \cap \bol{A}^t$ and $\bol{b}_1,\bol{b}_2 \in \partial B \cap \bol{B}^t$ such that the straight line segment $\overline{\bol{a}_1 \bol{a}_2}$ separates $\bol{b}_1$ and $\bol{b}_2$ in $\overline{B}$.  Let $\bol{A}_1$ and $\bol{A}_2$ be the components of $\bol{a}_1$ and $\bol{a}_2$, respectively, in $\mathfrak{A}^t$, and let $\bol{B}_1$ and $\bol{B}_2$ be the components of $\bol{b}_1$ and $\bol{b}_2$ in $\mathfrak{B}^t$.  Then $[\bol{A}_1 \cup \bol{A}_2] \cap [\bol{B}_1 \cup \bol{B}_2] = \0$ and by the remarks immediately following the definition of $\mathfrak{A}^t$ and $\mathfrak{B}^t$, each of these four components is either unbounded to the left or unbounded to the right.  Consider an arc $S$ in $\overline{B} \sm (\bol{B}_1 \cup \bol{B}_2)$ joining $\bol{a}_1$ and $\bol{a}_2$.  Then $\bol{A}_1 \cup \bol{A}_2 \cup S$ separates the plane into at least two components, and $\bol{B}_1$ and $\bol{B}_2$ must lie in different components of $\complex \sm (\bol{A}_1 \cup \bol{A}_2 \cup S)$.  It is then straightforward to see by considering cases that there exist $i,j \in \{1,2\}$ such that $\bol{B}_i$ lies above $\bol{A}_j$, a contradiction with \cref{lem:dichotomy stable}.
\end{proof}

For each $t \in [0,1]$, let $\bol{M}_t = \Equi(\bol{A}^t,\bol{B}^t)$.  In light of \cref{lem:noninterlaced}, $\bol{M}_t$ is a  $1$-manifold by \cref{thm:manifold}.

\begin{lem}
\label{M-disjoint}
For each $t \in [0,1]$ and each $n \in \mathbb{Z}$, $\bol{M}_t \cap \bol{E}_n\!\left( \frac{(1-\nu) |\Theta(b^t)|}{4\nu} \right) = \0$.

In particular, $\bol{M}_t \cap \bol{E}_n \!\left( \frac{|\Theta(b^t)|}{2} \right) = \0$.
\end{lem}

\begin{proof}
Let $n \in \mathbb{Z}$ and assume that $n \geq 0$ (the case $n < 0$ proceeds similarly).  Since $0<\nu<\frac{1}{3}$,  $\frac{(1-\nu) |\Theta(b^t)|}{4\nu}>\frac{|\Theta(b^t)|}{2}$, so $\bol{E}_n \!\left( \frac{|\Theta(b^t)|}{2} \right) \subset \bol{E}_n \!\left( \frac{(1-\nu) |\Theta(b^t)|}{4\nu} \right)$.

By \cref{uniqueB_n}, there is a component $\bol{C}$ of $\bol{X}$ such that $\bol{C}^t \cap E_n \!\left( \frac{|\Theta(b^t)|}{2} \right) \neq \0$ for all $t \in [0,1]$.  Since $n \geq 0$, $\bol{C} \subset \bol{A}$.

On the other hand, given any component $\bol{D}$ of $\bol{B}$, we have by \cref{uniqueB_n}(i) that $\bol{D}^t \cap \bol{E}_n \!\left( \frac{|\Theta(b^t)|}{2\nu} \right) = \0$ for all $t \in [0,1]$.  Thus $\bol{B}^t \cap \bol{E}_n \!\left( \frac{|\Theta(b^t)|}{2\nu} \right) = \0$ for all $t \in [0,1]$.  It follows that any point $x \in \bol{E}_n \!\left( \frac{(1-\nu) |\Theta(b^t)|}{4\nu} \right)$, the distance from $x$ to $\bol{A}^t$ is less than $\frac{|\Theta(b^t)|}{2} + \frac{(1-\nu) |\Theta(b^t)|}{4\nu} = \frac{(1+\nu) |\Theta(b^t)|}{4\nu}$, while the distance from $x$ to $\bol{B}^t$ is a greater than $\frac{|\Theta(b^t)|}{2\nu} - \frac{(1-\nu) |\Theta(b^t)|}{4\nu} = \frac{(1+\nu) |\Theta(b^t)|}{4\nu}$.  Thus $\bol{M}_t \cap \bol{E}_n \!\left( \frac{(1-\nu) |\Theta(b^t)|}{4\nu} \right) = \0$ for all $n$.
\end{proof}

\begin{lem}
\label{connM}
For each $t$ the set $\bol{M_t}$ is a connected 1-manifold.  Moreover, the vertical projection of $\bol{M}_t$ to the real axis $\real$ is onto.
\end{lem}

\begin{proof}
Since by \cref{lem:noninterlaced} $\bol{A}^t$ and $\bol{B}^t$ are non-interlaced, by \cref{thm:manifold}, $\bol{M}_t$ is a 1-manifold which separates $\bol{A}^t$ from $\bol{B}^t$.  By \cref{M-disjoint}, $\bol{M}_t$ is disjoint from $\bigcup_n \bol{E}_n \!\left( \frac{|\Theta(b^t)}{2} \right)$ and,  hence, $\bol{M}_t$ separates $\mathfrak{A}^t$ from $\mathfrak{B}^t$ (recall that $\mathfrak{A}^t$ and $\mathfrak{B}^t$ were defined above \cref{lem:vertical strip}).  Since all components of $\mathfrak{A}^t$ and $\mathfrak{B}^t$ are unbounded, no component of $\bol{M}_t$ is a simple closed curve and every component is a copy of $\real$ with both ends converging to infinity.  By \cref{lem:vertical strip} each end of a component of $\bol{M}_t$ either converges to $-\infty$ or $+\infty$. Fix $t$ and let $\bol{M}'$ be a component of $\bol{M}_t$.  Note that for all $\bol{x} \in \bol{M}'$ there exists a set of points $\bol{A}_x^t \subset \bol{A}^t$ closest to $x$ and $\bol{B}_x^t \subset \bol{B}^t$ closest to $x$ and that $\bigcup_{\bol{x} \in \bol{M}'} \bol{A}^t_x$ and $\bigcup_{x \in \bol{M}'} \bol{B}^t_x$ are separated by the line $\bol{M}'$.  For $x \in \bol{M}'$, let $r_x$ denote the distance from $x$ to $\bol{A}_x^t$ (equivalently, to $\bol{B}_x^t$).

If both ends of $\bol{M}'$ are unbounded to the same side, say on the  left side, then $\complex \sm \bol{M}'$ has two complementary components $P$ and $Q$, with $P$ only unbounded to the left (see \cref{fig:Mt projection}).  Assume that $\bigcup_{\bol{x} \in \bol{M}'} \bol{A}^t_x \subset P$ (the case $\bigcup_{\bol{x} \in \bol{M}'} \bol{B}^t_x \subset P$ is similar).  Note that since $P$ contains no components of $\bol{A}^t$ which are unbounded to the right, $P$ must contain components of $\bol{A}^t$ which are unbounded to the left.

\begin{figure}
\begin{center}
\includegraphics{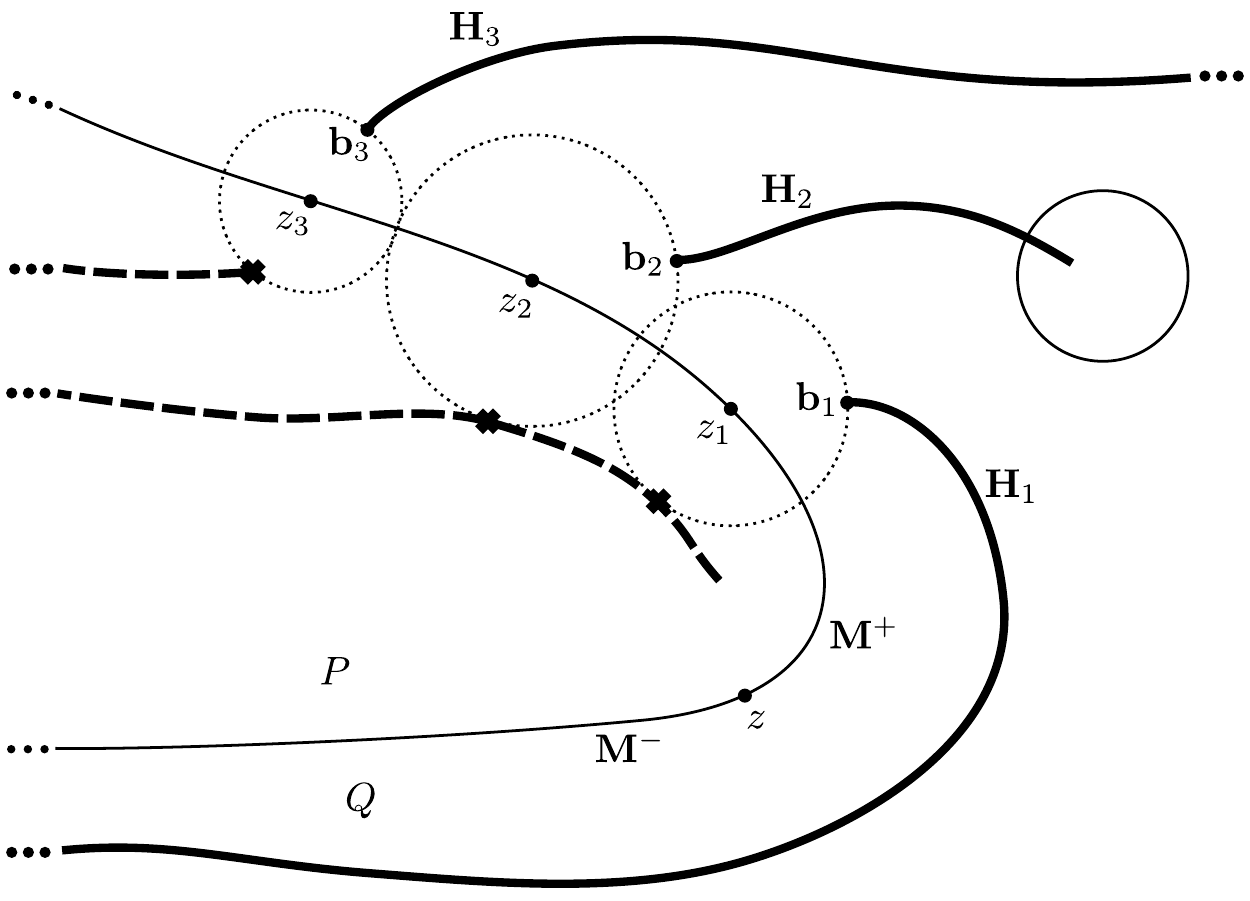}
\end{center}

\caption{An illustration of the situation described in the proof of \cref{connM}.}
\label{fig:Mt projection}
\end{figure}

Let $z \in \bol{M}'$.  Then $\bol{M}' \sm \{z\}$ consists of two rays $\bol{M}^+$ and $\bol{M}^-$ and we may assume that $\bol{M}^+$ lies above $\bol{M}^-$.  Choose $z_n \in \bol{M}^+$ monotonically converging to $-\infty$ and $\bol{b}_n \in \bol{B}^t_{z_n}$.  Since the radii $r_{z_n}$ are uniformly bounded, $\bol{b}_n$ also converges to $-\infty$.  Let $\bol{H}_n$ be the component of $\bol{B}^t$ that contains $\bol{b}_n$.

If $\bol{H}_n$ is unbounded to the left, by \cref{lem:dichotomy stable} it must lie below the unbounded components of $\bol{A}^t$ in $P$ and hence must ``go around'' $\bol{M}'$ as $\bol{H}_1$ does in Figure \cref{fig:Mt projection}.  If $\bol{H}_n$ is not unbounded to the left, then either it intersects some $\bol E_k(\frac{|\Theta(b^t)|}{2})$ for some $k < 0$ (as $\bol{H}_2$ does in \cref{fig:Mt projection}), or it is unbounded to the right (as $\bol{H}_3$ is in \cref{fig:Mt projection}).  In any case it is clear that there exists $c \in \real$ such that every component $\bol{H}_n$ intersects the vertical line $x = c$.

For each $n$ let $d_n$ be such that the point $(c,d_n) \in \bol{H}_n$.  By \cref{lem:vertical strip}, the sequence $d_n$ is bounded and, hence has an accumulation point $d_\infty$.  By \cref{compact in strip}, the component of $\bol{B}^t$ which contains $d_\infty$ is unbounded to the left, and clearly it lies above the unbounded components of $\bol{A}^t$ in $P$, a contradiction with \cref{lem:dichotomy stable}.  Hence, the vertical projection of $\bol{M}'$ to the real axis $\real$ is onto.

The proof that $\bol{M}_t = \bol{M}'$ is connected is similar and is left to the reader.
\end{proof}

\begin{lem}
\label{pathM}
For each $t \in [0,1]$, the set $\wE(\bol{M}_t) \cup \{0,1\}$ is the image of a path $\til{\gamma}_t$ in $\til{U}^t$ joining $0$ and $1$.
\end{lem}

\begin{proof}
Let $\imag$ denote the imaginary axis, so that $[-r,r] \times \imag$ is the strip in the plane between the vertical lines through $r$ and $-r$.  By \cref{lem:vertical strip}, for each $r > 0$, $\wE(([-r,r] \times \imag) \cap \bol{M}_t)$ is compact.  Together with \cref{connM}, this implies that we can choose a parameterization $\alpha: (0,1) \to \bol{M}_t$ so that:
\[ \lim_{s \to 0^+} \wE \circ \alpha(s) = \{0\} \]
and
\[ \lim_{s \to 1^-} \wE \circ \alpha(s) = \{1\} .\]
Define the path $\til{\gamma}_t: [0,1] \to \wE(\bol{M}_t) \cup \{0,1\}$ by $\til{\gamma}_t(s) = \wE \circ \alpha(s)$ for $s \in (0,1)$, and $\til{\gamma}_t(0) = 0$ and $\til{\gamma}_t(1) = 1$.  Then $\til{\gamma}_t$ is the required path.
\end{proof}

\subsection{Proof of \cref{main2}}
\label{sec:proof main2}

In this section we complete the proof of \cref{main2}.

Recall that $\e > 0$ is a fixed arbitrary number, and $0 < \nu < \frac{1}{3}$ has been chosen so that $\frac{8 \nu}{1 - \nu} < \frac{\e}{2}$.  Choose $0 < \delta < \frac{\e}{4}$ small enough so that the conclusions of \cref{uniqueB_n} and \cref{lem:dichotomy stable} hold (and therefore the results from \cref{sec:equi path} also hold).

For each $t \in [0,1]$, let $\gamma_t = (L^t \circ \Theta)^{-1} \circ \til{\gamma}_t$.  This $\gamma_t$ is a path in $U^t$ joining $0$ and $b^t$.

\begin{claim}
\label{paths small}
$\diam(\gamma_t([0,1])) < \e$ for all $t \in [0,1]$.
\end{claim}

\begin{proof}[Proof of \cref{paths small}]
\renewcommand{\qedsymbol}{\textsquare (\cref{paths small})}
By \cref{M-disjoint}, for all $t \in [0,1]$ and $n \in \mathbb{Z}$, $\bol{M}_t \cap \bol{E}_n \!\!\left( \frac{(1-\nu) |\Theta(b^t)|}{4\nu} \right) = \0$.

By \cref{ball-like}(ii), we have $\wE(\bol{M}_t) \subset B \!\left( 0, \frac{8\nu}{(1-\nu) |\Theta(b^t)|} \right)$.  Then \\ $(L^t)^{-1}(\wE(\bol{M}_t)) \subset B \!\left( 0, \frac{8\nu}{(1-\nu)} \right)$.  By the choice of $\nu$, and since $\Theta$ is a homeomorphism of $\complex$ which is the identity outside of $B(0, 2\delta) \subset B(0, \frac{\e}{2})$, it then follows that $\gamma_t([0,1]) = (L^t \circ \Theta)^{-1}(\wE(\bol{M}_t)) \subset B(0, \frac{\e}{2})$.
\end{proof}

\begin{claim}
\label{paths cont}
The sets $\gamma_t([0,1])$ vary continuously in the Hausdorff metric, and $\gamma_0$ is homotopic to $Q$ with endpoints fixed.
\end{claim}

\begin{proof}[Proof of \cref{paths cont}]
\renewcommand{\qedsymbol}{\textsquare (\cref{paths cont})}
By \cref{pathM}, $\til{\gamma}_t$ is a path in $\til{U}^t$ with endpoints $0$ and $1$.  To see that $\til{\gamma}_0$ is homotopic to $\til{Q} = (0,1)$ note first that since $\bol{A}^0$ is above the real axis and $\bol{B}^0$ is below the real axis, for each $(x,y) \in \bol{M}_0$ the vertical segment from $(x,0)$ to $(x,y)$ is disjoint from $\bol{X}^0$.  Hence we can construct a homotopy $k$ between $\bol{M}_0$ and $\real$ which fixes the x-coordinate of each point in $\bol{M}_0$ and decreases the absolute value of the $y$-coordinate to zero.  Then $\wE \circ k$ is the required homotopy between $\til{\gamma}_0$ and $\til{Q}$ with endpoints fixed.  Hence, $\gamma_0 = (L^0 \circ \Theta)^{-1} \circ \til{\gamma}_0$, is homotopic to $Q$ as required.

Suppose $t_i \to t_\infty$.  It is easy to see that $\limsup \bol{M}_{t_i} \subseteq \bol{M}_{t_\infty}$ by the definition of the equidistant sets $\bol{M}_t$.  Since, by \cref{connM}, each $\bol{M}_{t_i}$ and $\bol{M}_{t_\infty}$ is a connected $1$-manifold whose vertical projection to the real axis $\real$ is onto, it follows that $\liminf \bol{M}_{t_i} \supseteq \bol{M}_{t_\infty}$.  Thus $\lim \bol{M}_{t_i} = \bol{M}_{t_\infty}$.  It follows that $\gamma_t([0,1]) = (L^t \circ \Theta)^{-1} \circ \wE(\bol{M}_t)$ is continuous in the Hausdorff metric.
\end{proof}

Combined with \cref{liftexist}, Claims \ref{paths small} and \ref{paths cont} complete the verification of condition (ii) of \cref{main1technical}.  Therefore, by \cref{main1technical}, the isotopy $h$ of the compactum $X$ can be extended to the entire plane $\complex$.  This completes the proof of \cref{main2}.

\medskip
In \cref{main1} we have given necessary and sufficient conditions for an isotopy of a uniformly perfect compact set to extend to an isotopy of the plane.  These conditions involve the existence of an extension of the isotopy over sufficiently small crosscuts while controlling the size of the image.  The following problem remains open.

\bigskip
\begin{prob}
Are there intrinsic properties on $X$ and the isotopy $h$ of $X$, which do not involve the existence of extensions over small crosscuts, that characterize when an isotopy of $X$ can be extended over the plane?
\end{prob}

\bibliographystyle{amsalpha}
\bibliography{Isotopies}

\end{document}